\documentclass[a4paper, english, numberwithinsect]{article}
\usepackage{amsmath}
\usepackage{amsthm}
\usepackage{amssymb}
\usepackage{authblk}
\usepackage{hyperref}
\usepackage{mathrsfs}
\usepackage{mathtools}
\usepackage{microtype} 
\usepackage{url} 

\usepackage{todonotes}
\usepackage{subcaption}

\DeclareMathOperator{\arcsec}{arcsec}
\DeclareMathOperator*{\argmax}{arg \, max}
\DeclareMathOperator{\area}{area}
\DeclareMathOperator{\BD}{BD}
\DeclareMathOperator{\bd}{bd}
\DeclareMathOperator{\conv}{conv}

\newcommand{\R}{\mathbb{R}}

\usepackage{mathtools}
\DeclareMathOperator*{\dist}{dist}
\newcommand{\norm}[1]{\|#1\|}
\DeclareMathOperator{\VD}{VD}
\DeclareMathOperator{\VR}{VR}

\numberwithin{equation}{section}

\begin{document}

\theoremstyle{plain}
\newtheorem{proposition}{Proposition}[section]
\newtheorem{lemma}[proposition]{Lemma}
\newtheorem{theorem}[proposition]{Theorem}
\newtheorem{corollary}[proposition]{Corollary}

\theoremstyle{definition}
\newtheorem{definition}[proposition]{Definition}

\theoremstyle{remark}
\newtheorem{remark}[proposition]{Remark}

\title{On exact covering with unit disks}

\author[1]{Ji Hoon Chun}
\author[1]{Christian Kipp}
\author[1]{\mbox{Sandro Roch}}
\affil[1]{Technische Universit\"at Berlin\\ \texttt{lastname@math.tu-berlin.de}}


\maketitle

\begin{abstract}
  We study the problem of covering a given point set in the plane by unit disks so that each point is covered exactly once. We prove that~$17$ points can always be exactly covered. On the other hand, we construct a set of~$657$ points where an exact cover is not possible. 
\end{abstract}


\section{Introduction}

In 2008, Inaba~\cite{Inaba2008_1} introduced the following puzzle about covering sets of points in the plane: 

\vspace{1ex}

\noindent \textit{Show that any set of~$\mathit{10}$~points in~$\mathbb{R}^{2}$~can be covered by nonoverlapping unit disks.} 

\vspace{1ex}

\noindent Inaba solved this puzzle~\cite{Inaba2008_2, Winkler2010} with an elegant probabilistic argument (see Appendix~\ref{sec: Inaba proof} for a deterministic version). In this article, we study a relaxed version of this covering problem. Given a point set~$X\subset\mathbb{R}^2$, can we find a family~$\mathscr{D}$ of not necessarily disjoint unit disks so that each point~$\mathbf{x} \in X$ is contained in exactly one disk~$D\in\mathscr{D}$? We call such a family an \textit{exact cover} of~$X$. For example, in Figure~\ref{fig:example_primal}, the two red disks form an exact cover of the four blue points.

\begin{figure}[tb]
    \centering
    \begin{subfigure}[b]{0.49\textwidth}
        \centering
        \includegraphics[scale=0.85, page=1]{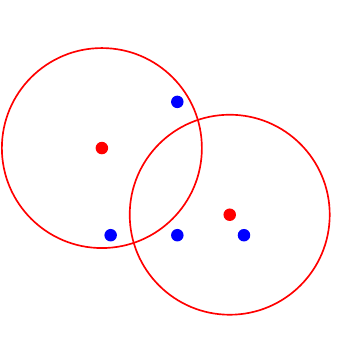}
        \caption{}
        \label{fig:example_primal}
    \end{subfigure}
    \hfill
    \begin{subfigure}[b]{0.49\textwidth}
        \centering
        \includegraphics[scale=0.85, page=2]{figures/primal_dual_example.pdf}
	\caption{}
	\label{fig:example_dual}
    \end{subfigure}
	
    \caption{\textit{Left:} Primal solution (exact covering set). \textit{Right:} Dual solution (exact hitting set).}
    \label{fig:example_primal_dual}
\end{figure}

Let~$B^{2}\coloneqq\left\{ \mathbf{x}\in\mathbb{R}^{2}\,\middle|\,\left\Vert \mathbf{x}\right\Vert <1\right\}$, where~$\norm{\cdot}$ denotes the Euclidean norm. We define an (open) \textit{disk} with \textit{center}~$\mathbf{c}\in\mathbb{R}^2$ and \textit{radius}~$r > 0$ as the set~$D_{\mathbf{c}, r} \coloneqq \mathbf{c} + r B^{2}$; if~$r = 1$ we call it the \textit{unit disk} and write~$D_{\mathbf{c}}$. 

\begin{definition}\label{def: disjoint and overlapping disk covering}
    Let~$\sigma_{2}$ be the largest~$n \in \mathbb{N}$ such that any set of~$n$ points in the plane can be covered by disjoint unit disks. Let~$\widehat{\sigma}_{2}\in \mathbb{N}$ be the corresponding number for the relaxed problem involving exact covers.
\end{definition}

As a covering using disjoint disks is also an exact covering, we have the basic relationship of~$\widehat{\sigma}_{2} \geq \sigma_{2}$. The current best known bounds for~$\sigma_{2}$ are~$12 \leq \sigma_{2} \leq 44$~\cite{AloupisHearnIwasawaUehara2012}. Aloupis, Hearn, Iwasawa, and~Uehara~(2012)~\cite{AloupisHearnIwasawaUehara2012} improved Inaba's lower bound to~$\sigma_{2} \geq 12$ through a careful analysis of the probabilistic method on one-dimensional slices of the plane. In the other direction,~$\sigma_{2}$ is finite: Intuitively, with a dense enough arrangement of points, this problem becomes similar to the problem of covering the entire set~$\conv X$, which is impossible using disjoint disks. Specific upper bounds were reduced in rapid succession from~$\sigma_{2} < 60$ by Winkler~(2010)~\cite{Winkler2010} to~$\sigma_{2} < 55$ by Elser~(2011)~\cite{Elser2011} and~$\sigma_{2} < 53$ by Okayama, Kiyomi, and~Uehara~(2012)~\cite{OkayamaKiyomiUehara2012}. Most recently, Aloupis, Hearn, Iwasawa,~and Uehara~(2012)~\cite{AloupisHearnIwasawaUehara2012} proved~$\sigma_{2} < 50$ ``by hand'' and demonstrated~$\sigma_{2} < 45$ using computer calculations. 


\subsection{Results}

In Section~\ref{sec: lower bound}, we build on some of the mentioned works on lower bounds to establish the following lower bound on~$\widehat{\sigma}_{2}$:

\begin{theorem}\label{thm: lower bound}
    We have~$\widehat{\sigma}_{2} \geq 17$. 
\end{theorem}

The finiteness of~$\widehat{\sigma}_{2}$ can be deduced by a similar argument as for the finiteness of~$\sigma_{2}$. 
In Section~$\ref{sec: upper bound}$ we construct a close arrangement of points that cannot be exactly covered, leading to the following (rather weak) upper bound on~$\widehat{\sigma}_{2}$:

\begin{theorem}\label{thm: upper bound}
    We have~$\widehat{\sigma}_{2} < 657$. 
\end{theorem}

For the full proofs of Theorem~\ref{thm: lower bound} and Theorem~\ref{thm: upper bound} we refer to Sections~\ref{sec: lower bound} and~\ref{sec: upper bound} respectively. 


\subsection{Relation between exact covering and exact hitting}\label{subsec: relation between exact covering and exact hitting}

We denote by~$\mathscr{X}$ the collection of all finite point sets in~$\mathbb{R}^{2}$. A point~$\mathbf{x} \in \mathbb{R}^{2}$ is contained in a unit disk~$D_{\mathbf{c}}$ centered at~$\mathbf{c} \in \mathbb{R}^{2}$ if and only if~$\mathbf{c}$ is contained in the unit disk~$D_{\mathbf{x}}$ centered at~$\mathbf{x}$. By this simple observation, the problem of exactly covering some given~$X \in \mathscr{X}$ by unit disks (\textit{primal problem}) becomes equivalent to the following \textit{dual problem}: Let~$\mathscr{D}_X \coloneqq \left\{D_{\mathbf{x}} \,\middle|\, \mathbf{x} \in X\right\}$; find a~$P \in \mathscr{X}$ such that each disk~$D \in \mathscr{D}_X$ contains exactly one point~$\mathbf{p} \in P$. See Figure~\ref{fig:example_dual} for an example of the dual problem. In the literature, such a set~$P$ is also called an \textit{exact hitting set}. A dual solution~$P$ yields the solution~$\mathscr{D} \coloneqq \left\{D_{\mathbf{p}} \,\middle|\, \mathbf{p} \in P\right\}$ to the exact covering problem. Vice versa, a solution~$\mathscr{D}$ to the exact covering problem gives a solution to the dual problem by taking the disk centers. 

In the dual perspective, the boundary circles of the disks~$\mathscr{D}_X$ decompose the plane into \textit{cells}. Observe that all points in a given cell are contained in the same set of disks, so the exact position of a dual solution point~$\mathbf{p} \in P$ is irrelevant. Hence, it is sufficient to select a set of cells so that each~$D\in\mathscr{D}_X$ contains exactly one selected cell. In the example of Figure~\ref{fig:example_dual}, the two red shaded cells form a solution. This observation shows that the solution space for the dual problem and for the exact covering problem is in fact discrete, and methods such as Knuth's Algorithm X (see~\cite{Knuth2000} or Section 7.2.2.1 in~\cite{KnuthTAOCP4B2023}), integer programming, or SAT solvers (see~\cite{junttilaKaski10}) can be used. 


\section{A lower bound}\label{sec: lower bound}

We exclude the following trivial case from our proofs in this section: If~$X$ lies on a line then~$X$ can be covered by disjoint disks. Denote by~$\mathscr{X}'$ the subset of~$\mathscr{X}$ that excludes every point set on a line. To prove Theorem~\ref{thm: lower bound}, we have to show that all~$X \in \mathscr{X}'$ with~$\lvert X\rvert \leq 17$ can be exactly covered. We combine three separate components on top of Inaba's original probabilistic proof. In Subsection~\ref{subsec: boundary points} we show that~$\widehat{\sigma}_{2} \geq \sigma_{2} + 4$. In Subsection~\ref{subsec: parametric covering} we obtain~$\widehat{\sigma}_{2} \geq 16$ using a covering version of Betke, Henk, and Wills's parametric density~\cite{BetkeHenkWills1994}, and~$\widehat{\sigma}_{2} \geq 17$ by showing that in some cases, a disk~$D$ that overlaps with~$\conv X$ can be removed from an exact cover~$\mathscr{D}$ of~$X$ so that~$\mathscr{D} \left\backslash\, \left\{D\right\}\right.$ is still an exact cover of~$X$. 

In this section, we denote the vertices of~$\conv X$ by~$\mathbf{v}^{1}, \ldots, \mathbf{v}^{k}$. We use subscripts to denote the coordinates of a vector, so for example~$\mathbf{v}^{i} = \begin{pmatrix} v_{1}^{i} \\ v_{2}^{i} \end{pmatrix}$. 


\subsection{Boundary points}\label{subsec: boundary points}

\begin{definition}\label{def: boundary point}
    Let~$X \in \mathscr{X}$ and~$\mathbf{v} \in X$. The point~$\mathbf{v}$ is a \textit{boundary point of~$X$} if~$\mathbf{v}$ is on the boundary of~$\conv X$. 
\end{definition}

Let~$X \in \mathscr{X}$ and~$\mathbf{v}^{1}, \ldots, \mathbf{v}^{k}$ be the boundary points of~$X$. L\'aszl\'o Kozma (private communication) observed that a covering~$\mathscr{D}'$ of the non-boundary points~$X' \coloneqq X \left\backslash\,\left\{\mathbf{v}^{1}, \ldots, \mathbf{v}^{k}\right\}\right.\!$ by \textit{disjoint} disks can always be extended to an exact cover of~$X$. A boundary point~$\mathbf{v}^i$ is covered by at most one disk in~$\mathscr{D}'$ because the disks are disjoint. If~$\mathbf{v}^i$ is not already covered by~$\mathscr{D}'$, then it can be covered by a new disk which contains~$\mathbf{v}^i$ but no other point of~$X$. The resulting disk configuration yields an exact cover~$\mathscr{D}$ of~$X$ (Figure~\ref{fig:kozma_method}). In particular, if~$\left|X\right| \leq \sigma_{2} + k$ then~$X$ can be exactly covered. We refer to this strategy as the \textit{Extension Argument}: 

\begin{lemma}[Extension Argument]\label{lem: Kozma method}
    Let~$X \in \mathscr{X}$ and~$k$ be the number of boundary points of~$\conv X$. 
    \begin{enumerate}
        \item If~$\left\vert X \right\vert \leq \sigma_{2} + k$ then~$X$ can be exactly covered. 
        \item If~$k \leq 2$ then~$X$ can be exactly covered regardless of~$\left\vert X \right\vert$. 
        \item We have~$\widehat{\sigma}_{2} \geq \sigma_{2} + 3$. 
    \end{enumerate}
\end{lemma}

\begin{figure}
    \centering
    \includegraphics{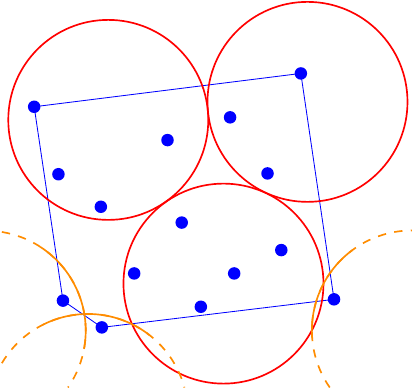}
    \caption{Proof idea of the Extension~Argument~\ref{lem: Kozma method}. We extend the red disjoint disk covering of the non-boundary points by adding a new orange disk at each uncovered boundary point. The new disks may overlap each other or existing disks.}
    \label{fig:kozma_method}
\end{figure}

\begin{proof}
    As~$X$ can be assumed to not lie on a line, we have~$k \geq 3$. Let~$X \in \mathscr{X}$ with~$\left|X\right| \leq \sigma_{2} + k$. The set~$X' \in \mathscr{X}'$ has at most~$\sigma_{2}$ points, so there exists a collection~$\mathscr{D}'$ of disjoint disks that covers~$X'$. 
    
    For each vertex~$\mathbf{v}^{i}$ of~$X$, consider the two cases below. 
    \begin{enumerate}
        \item If~$\mathbf{v}^{i}$ is already covered by some disk in~$\mathscr{D}'$, then do nothing. 
        \item If~$\mathbf{v}^{i}$ is not covered by any disk in~$\mathscr{D}'$, then add a new disk that covers~$\mathbf{v}^{i}$ but no other point of~$X$. Such a disk can always be found because~$\conv X$ is convex and~$B^{2}$ is strictly convex, but it may overlap with some existing disk(s) in~$\mathscr{D}'$ (Figure~\ref{fig:kozma_method}). 
    \end{enumerate}
    The cover~$\mathscr{D} \coloneqq \mathscr{D}' \cup \left\{\text{all new disks from case 2}\right\}$ is an exact cover of~$X$. 
    
    The third part of the Extension Argument is a direct consequence of the first two parts. 
\end{proof}

\begin{remark}
    The Extension~Argument~\ref{lem: Kozma method}, along with its generalization below (Lemma~\ref{lem: Kozma method, improved}), cannot be applied to extend an \textit{exact} cover~$\mathscr{D}'$ of~$X'$ to an exact cover of~$X$. In this case, the disks of~$\mathscr{D}'$ are allowed to overlap, so some~$\mathbf{v}^{i}$ may be contained in more than one disk of~$\mathscr{D}'$.
\end{remark} 

The Extension Argument improves the basic inequality~$\widehat{\sigma}_{2} \geq \sigma_{2}$ to~$\widehat{\sigma}_{2} \geq \sigma_{2} + 3$. Hence this lemma, combined with Inaba's and Aloupis, Hearn, Iwasawa, and Uehara's results, obtains~$\widehat{\sigma}_{2} \geq 13$ and~$\widehat{\sigma}_{2} \geq 15$ respectively. These lower bounds are limited by the case where~$\conv X$ is a triangle, since otherwise~$X$ has at least four boundary points. Therefore, we wish to relax Definition~\ref{def: boundary point} so that \textit{every}~$X \in \mathscr{X}'$ has at least four ``generalized boundary points'' that behave like boundary points, or can be exactly covered regardless of~$\left\vert X \right\vert$. 

\begin{definition}
    Let~$X \in \mathscr{X}$ and~$\mathbf{b} \in X$. The point~$\mathbf{b}$ is a \textit{generalized boundary point} of~$X$ if there exists a~$\mathbf{c} \in \mathbb{R}^{2}$ such that~$X \cap D_{\mathbf{c}} = \left\{\mathbf{b}\right\}$. 
\end{definition}

All vertices and boundary points of~$X$ are generalized boundary points of~$X$. Let~$\mathbf{b}^{1}, \ldots, \mathbf{b}^{k}$ be the generalized boundary points of~$X$, and write~$X'' \coloneqq X \left\backslash\, \left\{\mathbf{b}^{1}, \ldots, \mathbf{b}^{k}\right\} \right.$. 

Let~$k \in \mathbb{N}_{0}$. For convenience, we write~$\widehat{\sigma}_{2}\left(k\right)$ to denote the largest~$n \in \mathbb{N}$ such that all sets in~$\mathscr{X}$ with at most~$n$ points, at least~$k$ of which are generalized boundary points, can be exactly covered. For any~$k' \geq k \geq 0$, we have~$\widehat{\sigma}_{2}\left(k'\right) \geq \widehat{\sigma}_{2}\left(k\right) = \widehat{\sigma}_{2}\left(0\right) = \widehat{\sigma}_{2}$. 

\begin{lemma}[Generalized Extension Argument]\label{lem: Kozma method, improved}
    Let~$X \in \mathscr{X}'$ and~$k$ be the number of generalized boundary points of~$\conv X$. 
    \begin{enumerate}
        \item If~$\lvert X \rvert \leq \sigma_{2} + k$ then~$X$ can be exactly covered. (That is,~$\widehat{\sigma}_{2}\left(k\right) \geq \sigma_{2} + k$.) 
        \item If~$k \leq 3$ then~$X$ can be exactly covered regardless of~$\left\vert X \right\vert$. 
        \item We have~$\widehat{\sigma}_{2} \geq \sigma_{2} + 4$. 
    \end{enumerate}
\end{lemma}

Thanks to our definition of a generalized boundary point, the proof of Part 1 of the Generalized Extension Argument is fundamentally the same as the proof of the corresponding claim in the Extension Argument. The main difference is that to avoid a generalized boundary point being covered by more than one new disk, we add these new disks one at a time instead of all at once. 

\begin{proof}[Proof of the Generalized Extension Argument, Part 1]
    Let~$\mathbf{b}^{1}, \ldots, \mathbf{b}^{k}$ be the generalized boundary points of~$X$. Let~$X_{0} \coloneqq X''$ and~$X_{i} \coloneqq X_{i - 1} \cup \left\{\mathbf{b}^{i}\right\}$ for~$i \in \left\{1, \ldots, k\right\}$. The set~$X_{0} \in \mathscr{X}'$ has at most~$\sigma_{2}$ points, so there exists a disjoint covering~$\mathscr{D}_{0}$ of~$X_{0}$. 

    We recursively show that each~$X_{i}$ can be exactly covered. Let~$i \in \left\{1, \ldots, k\right\}$. Suppose that we have an exact cover~$\mathscr{D}_{i - 1} \supseteq \mathscr{D}_{0}$ of~$X_{i - 1}$. Consider the generalized boundary point~$\mathbf{b}^{i}$. 
    \begin{enumerate}
        \item If~$\mathbf{b}^{i}$ is covered by some disk in~$\mathscr{D}_{i - 1}$, then do nothing. The collection~$\mathscr{D}_{i} \coloneqq \mathscr{D}_{i - 1}$ is already an exact cover of~$X_{i}$. 
        \item If~$\mathbf{b}^{i}$ is not covered by any disk in~$\mathscr{D}_{i - 1}$, then add a new disk~$D_{i}$ that covers~$\mathbf{b}^{i}$ but no other point of~$X$. Such a disk can always be found by the definition of a generalized boundary point. The collection ~$\mathscr{D}_{i} \coloneqq \mathscr{D}_{i - 1} \cup \left\{D_{i}\right\}$ is an exact cover of~$X_{i}$. 
    \end{enumerate}
    Setting~$i = k$ finishes the proof. 
\end{proof}

The Generalized Extension Argument combined with Aloupis, Hearn, Iwasawa, and Uehara's lower bound implies~$\widehat{\sigma}_{2} \geq 16$. 

For the remainder of this subsection, assume that~$X \in \mathscr{X}'$ has at least four points, since if~$X$ has three or fewer points, then there is nothing to prove. If~$\conv X$ is a triangle, then we write~$T \coloneqq \conv X$ and~$T' \coloneqq T\left\backslash \, \left\{ \mathbf{v}^{1},\mathbf{v}^{2},\mathbf{v}^{3}\right\} \right.$, where~$\mathbf{v}^{1}$,~$\mathbf{v}^{2}$, and~$\mathbf{v}^{3}$ are its three vertices ($X$ may have additional boundary points). Let~$R_{T}$ be the circumradius of~$T$. 

Our proof of Part 2 of the Generalized Extension Argument consists of a more involved case analysis than for the Extension Argument. Let~$X \in \mathscr{X}'$. The following cases envelop all the possibilities for a triangular convex hull~$\conv X$, ordered by descending side lengths and circumradius. 

\begin{lemma}\label{lem: large triangles have four generalized boundary points}
    If at least one side of~$T$ has length at least~$2$, then~$X$ has at least four generalized boundary points. (In this case,~$R_{T} \geq 1$.) 
\end{lemma}

\begin{lemma}\label{lem: medium triangles have four generalized boundary points}
    If all sides of~$T$ have length less than~$2$ and~$R_{T} > 1$, then~$X$ has at least four generalized boundary points. 
\end{lemma}

\begin{lemma}\label{lem: "unit" triangles have four generalized boundary points}
    If all sides of~$T$ have length less than~$2$ and~$R_{T} = 1$, then~$X$ can be covered by four unit disks. 
\end{lemma}

\begin{proof}
    The interior of the circumcircle~$D$ of~$T$ is a unit disk that covers~$T'$. Each vertex of~$T$ can be covered by a separate unit disk as in the proof of the Extension~Argument~\ref{lem: Kozma method}. 
\end{proof}

\begin{lemma}\label{lem: small triangles have four generalized boundary points}
    If all sides of~$T$ have length less than~$2$ and~$R_{T} < 1$, then~$X$ can be covered by one unit disk. 
\end{lemma}

\begin{proof}
    Cover the circumcircle of~$T$ with a unit disk. 
\end{proof}

The first two lemmas show that for a ``large'' triangle we can always find one additional generalized boundary point~$\mathbf{b} \in X'$. This point is, roughly speaking, the closest non-vertex of~$X$ to the longest edge of~$\conv X$. Since the proofs are lengthy, we leave them for later in this section. A ``small'' triangle~$T$ lacks room for a fourth generalized boundary point since its vertices are too close to each point of~$X$ (including the other vertices). However, the other two lemmas show that every small triangle can be exactly covered by a fixed number of disks independently of~$\left\vert X \right\vert$. 

\begin{proof}[Proof of the Generalized Extension Argument~\ref{lem: Kozma method, improved}, Parts 2 and 3]
    Since~$X \in \mathscr{X}'$ has~$k \leq 3$ generalized boundary points, its convex hull of~$X$ must be a triangle. If~$R_{T} > 1$ then~$X$ has four generalized boundary points by Lemmas~\ref{lem: large triangles have four generalized boundary points}~and~\ref{lem: medium triangles have four generalized boundary points}, so~$R_{T} \leq 1$. So~$X$ can be exactly covered by Lemmas~\ref{lem: "unit" triangles have four generalized boundary points}~and~\ref{lem: small triangles have four generalized boundary points}.

    As before, Part 3 of the Generalized Extension Argument is a direct consequence of Parts 1 and 2. 
\end{proof}

The main idea for the proof of Lemma~\ref{lem: large triangles have four generalized boundary points} is as follows. Suppose that the longest side~$\overline{\mathbf{v}^{1}\mathbf{v}^{2}}$ of~$T$ has length greater than or equal to~$2$. Place a unit disk~$D$ outside~$T$ and push it through~$\overline{\mathbf{v}^{1}\mathbf{v}^{2}}$ from below, without touching either~$\mathbf{v}^{1}$ or~$\mathbf{v}^{2}$, until it reaches a point~$\mathbf{b} \in X'$. If chosen correctly,~$D$ will cover~$\mathbf{b}$ but no other point of~$X$, hence~$\mathbf{b}$ is the desired fourth generalized boundary point. 

We specify how this generalized boundary point is chosen using the following technical definition. 

\begin{definition}\label{def: bulldozer}
    Let~$X \in \mathscr{X}'$ satisfy the hypothesis of Lemma~\ref{lem: large triangles have four generalized boundary points}. Without loss of generality, position~$X$ so that~$\overline{\mathbf{v}^{1}\mathbf{v}^{2}}$ is the longest side of~$T = \conv X$,~$\mathbf{v}^{1}=\begin{pmatrix}-v \\ \hphantom{+}0\end{pmatrix}$ and~$\mathbf{v}^{2}=\begin{pmatrix}v \\ 0\end{pmatrix}$ for some~$v \geq 1$, and~$v_{2}^{3} > 0$. Let~$t \geq -1$. Define the \textit{bulldozer}\footnote{The term ``bulldozer'' and its abbreviation are from~\cite{AMDBulldozer2007}, in particular slide 23.}~$\BD\left(X,t\right)$ by (see Figure~\ref{fig: bulldozer 1}) 
    \[
        \BD\left(X,t\right)\coloneqq\bigcup_{c\in\left[-\left(v-1\right),v-1\right]}\left(\begin{pmatrix}c \\ t\end{pmatrix}+B^{2}\right).
    \]
\end{definition}

We will show that if~$\overline{\mathbf{v}^{1}\mathbf{v}^{2}}$ has length at least~$2$, i.e.,~$v \geq 1$, then by increasing~$t$, the bulldozer will eventually encounter a generalized boundary point of~$X$ distinct from its vertices. 

\begin{proof}[Proof of Lemma~\ref{lem: large triangles have four generalized boundary points}]
    Without loss of generality, position~$T$ as in the hypotheses of Definition~\ref{def: bulldozer}. 
    
    Since~$\overline{\mathbf{v}^{1}\mathbf{v}^{2}}$ is the longest side of~$T$, it follows that~$-v < v_{1}^{3} < v$ and 
    \begin{equation}
        T \left\backslash\, \left\{ \mathbf{v}^{1},\mathbf{v}^{2},\mathbf{v}^{3}\right\} \right.\subset\bigcup_{t\geq-1}\BD\left(X,t\right). \label{eq: the bulldozer bulldozes everything}
    \end{equation}
    In other words, the bulldozers~$\BD\left(X, t\right)$ can capture all points in~$T'$. Since~$T'$ contains at least one point of~$X$, there exists a point~$\mathbf{b} = \begin{pmatrix} b_{1} \\ b_{2} \end{pmatrix}$ on the boundary of~$\BD\left(X, t_{\max}\right)$ that is not~$\mathbf{v}^{1}$,~$\mathbf{v}^{2}$, or~$\mathbf{v}^{3}$. 

    We claim that~$\mathbf{b}$ is a generalized boundary point of~$X$. Let 
    \[
        t_{\max}\coloneqq\sup\left\{t \geq -1 \,\middle|\, \BD\left(X, t\right)\cap X = \emptyset\right\},
    \]
    which measures the position of the highest possible bulldozer~$\BD\left(X, t_{\max}\right)$ which does not intersect any points of~$X$. Consider the disk
    \[
        D\left(\varepsilon\right) \coloneqq \begin{pmatrix} b_{1} \\ t_{\max} + \varepsilon \end{pmatrix} + B^{2} \subseteq \BD\left(X,\, t_{\max} + \varepsilon\right).
    \]
    For a small enough positive~$\varepsilon$, this disk contains~$\mathbf{b}$ but does not intersect any other point of~$X$. 
\end{proof}

\begin{figure}
    \centering
    \includegraphics[width=4.75in]{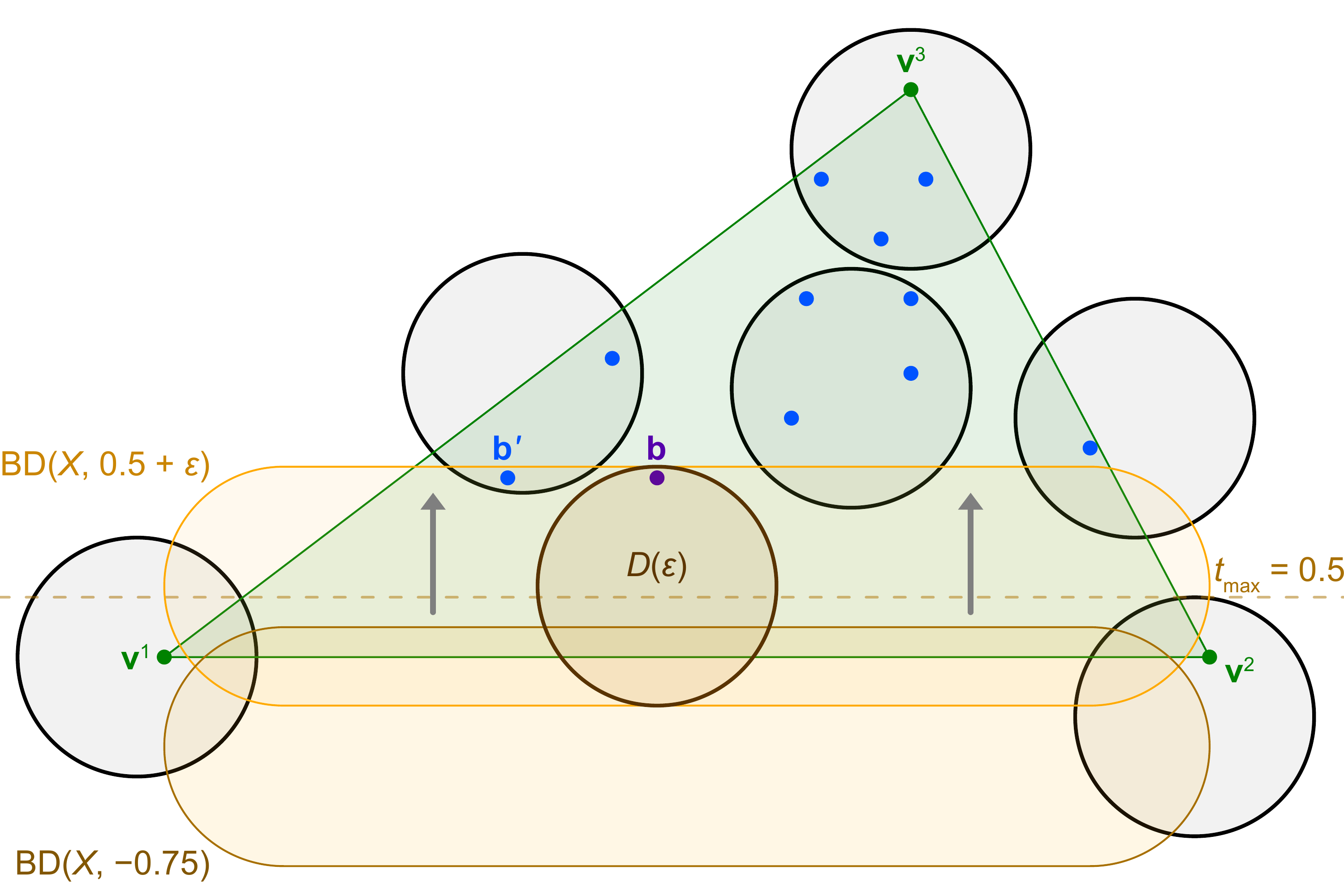}
    \caption{How to find the fourth generalized boundary point in the proof of Theorem~\ref{lem: large triangles have four generalized boundary points}. The bulldozers are~$\BD\left(X, -0.75\right)$ (medium orange) and~$\BD\left(X, t_{\max} + \varepsilon\right)$ (light orange), with~$t_{\max} = 0.5$. The generalized boundary point~$\mathbf{b} = \begin{pmatrix} -0.25 \\ \hphantom{+}1.5\hphantom{0} \end{pmatrix} \in X'$ is in dark purple and the disk~$D\left(\varepsilon\right) = \begin{pmatrix} -0.25 \\ 0.5 + \varepsilon \end{pmatrix} + B^{2}$ in dark orange. The grey arrows represent the intuition of finding~$\mathbf{b}$ by moving the bulldozer ``upwards'' until it reaches a point of~$X'$. \\
    The point~$\mathbf{b}'$ is on the boundary of~$\BD\left(X, t_{\max}\right)$, so it is also a generalized boundary point of~$X$ which we could have chosen instead of~$\mathbf{b}$.}
    \label{fig: bulldozer 1}
\end{figure}

\begin{remark}
    The~$y$-coordinate~$b_{2}$ found in the above proof is usually equal to~$t_{\max} + 1$. However, if~$\mathbf{b}$ lies on the curved part of~$\BD\left(X, t_{\max}\right)$ then~$t_{\max} < b_{2} < t_{\max} + 1$. 
\end{remark}

\begin{remark}
    If~$\mathbf{b}$ is just defined as a point of~$X'$ with smallest~$y$-coordinate, then it is not necessarily a generalized boundary point. Figure~\ref{fig: bulldozer 3} illustrates a scenario in which this point~$\mathbf{b}$ lies very close to~$\mathbf{v}^{2}$, and the given disjoint cover of~$X' \left\backslash\, \left\{\mathbf{b}\right\} \right.\!$ cannot be extended to an exact cover of~$X$. 
    
    \begin{figure}[tb]
        \centering
        \begin{subfigure}[b]{0.49\textwidth}
            \centering
            \includegraphics[width=2.3in]{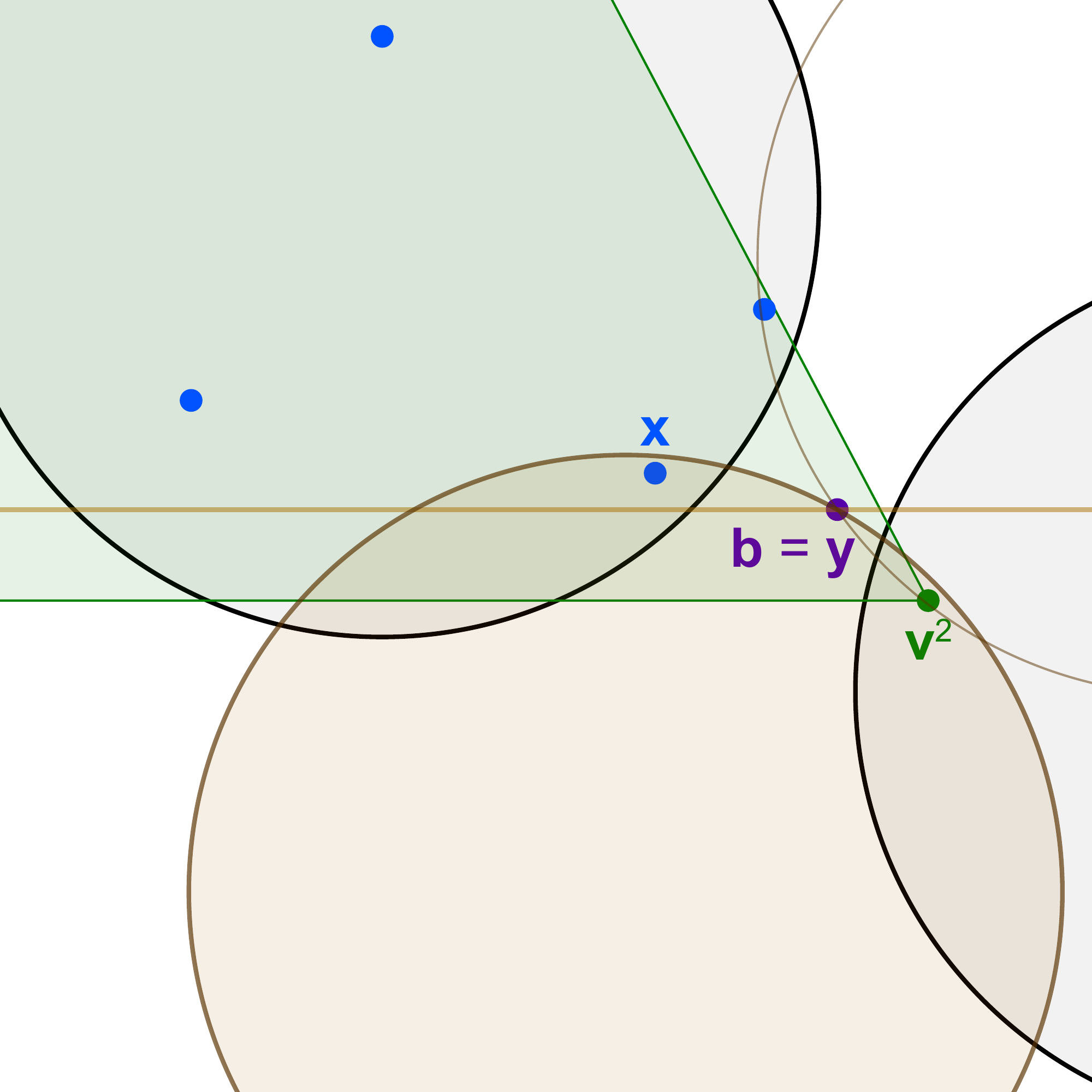}
            \caption{}
            \label{fig: bulldozer 3}
        \end{subfigure}
        \hfill
        \begin{subfigure}[b]{0.49\textwidth}
            \centering
            \includegraphics[width=2.3in]{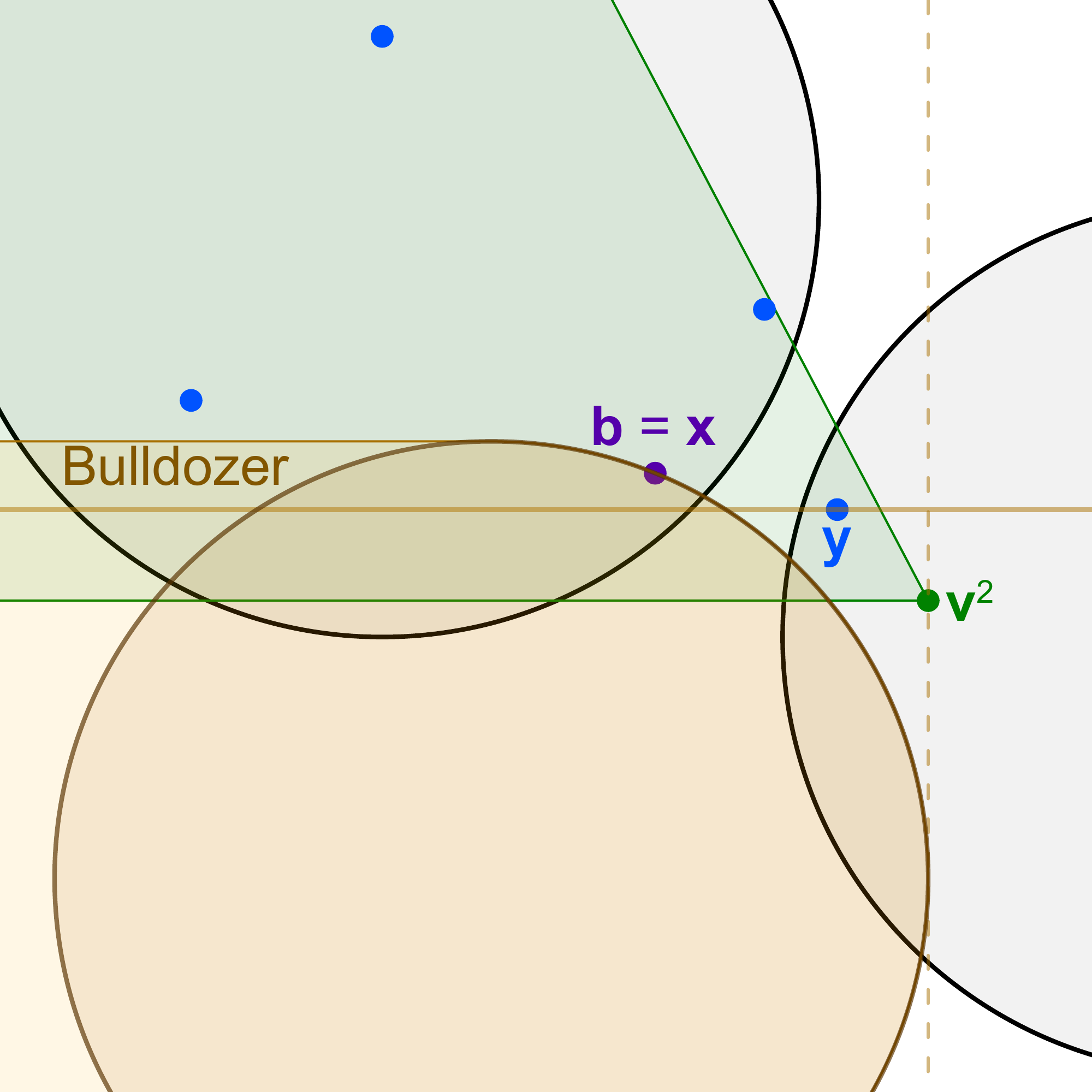}
            \caption{}
            \label{fig: bulldozer 4}
        \end{subfigure}
        \caption{An example that illustrates the need for the complicated definition of the bulldozer. The grey disks are an exact cover of~$X \left\backslash\, \left\{\mathbf{b}\right\}\right.$ while the orange disk~$D\left(\varepsilon\right)$ (see the proof of Lemma~\ref{lem: large triangles have four generalized boundary points}) is from the bulldozer. \\ 
        \textit{Left:} The point~$\mathbf{b} = \mathbf{y}$ has smallest~$y$-coordinate among all points in~$X'$. In this case,~$\mathbf{b}$ cannot be covered by a disk without~$\mathbf{x}$ or~$\mathbf{v}^{2}$ being covered by \textit{two} disks. (The faint brown disk at the top right shows that~$\mathbf{b}$ cannot be covered in this way by approaching from another edge of~$T$.) \\ 
        \textit{Right:} The definition of the bulldozer fixes the problem in Figure~\ref{fig: bulldozer 3}. The bulldozer finds a different point~$\mathbf{b} = \mathbf{x}$, and this choice of~$\mathbf{b}$ results in a different disjoint cover~$\mathscr{D}'$ of~$X \left\backslash\, \left\{\mathbf{b}\right\} \right.\!$. \\
        If~$\mathbf{b}$ is not covered by a disk in~$\mathscr{D}'$ then~$\mathscr{D}' \cup \left\{D\left(\varepsilon\right)\right\}$ is an exact cover of~$X$. \\
        If~$\mathbf{b}$ is already covered by a disk in~$\mathscr{D}'$, as in Figure~\ref{fig: bulldozer 4}, then~$\mathscr{D}'$ is an exact cover of~$X$.}
        \label{fig: bulldozer 1 4}
    \end{figure}
\end{remark}

If all three sides of~$T$ have lengths less than~$2$, then~$\BD\left(X, t\right) = \emptyset$, so a different approach is needed using the same idea. Let~$\mathbf{h}$ be the orthocenter of~$T$. For each~$1 \leq i < j \leq 3$, let~$D_{i, j}$ be the disk whose boundary passes through~$\mathbf{v}^{i}$,~$\mathbf{v}^{j}$, and~$\mathbf{h}$. The points~$\mathbf{v}^{1}$,~$\mathbf{v}^{2}$,~$\mathbf{v}^{3}$, and~$\mathbf{h}$ form an orthocentric system, and it is well known (e.g.~\cite{Wells1991}, page 165) that the radii of all the~$D_{i, j}$ are equal to~$R_{T}$. We only require~$R_{T} > 1$ for Lemma~\ref{lem: medium triangles have four generalized boundary points}, but we also explain the other cases~$R_{T} = 1$ and~$R_{T} < 1$ below for context. 

For each~$1 \leq i < j \leq 3$, let~$\BD_{i, j}$ be the \textit{unit} disk whose boundary passes through~$\mathbf{v}^{i}$ and~$\mathbf{v}^{j}$ and whose center is on the same side of the line~$\overline{\mathbf{v}^{i}\mathbf{v}^{j}}$ as the center of~$D_{i, j}$. For convenience, we will also refer to the disks~$\BD_{i, j}$ as bulldozers---see Figure~\ref{fig: circumcircle 4 1} for a comparison of~$D_{i, j}$ and~$\BD_{i, j}$. We describe the relationship between~$D_{i, j}$ and~$\BD_{i, j}$ for various values of~$R_{T}$ below. 
\begin{enumerate}
    \item If~$R_{T} > 1$, then~$\BD_{i, j}$ is smaller than~$D_{i, j}$ for all~$1 \leq i < j \leq 3$. Then each~$\BD_{i, j}$ can be pushed ``further into''~$T$ than its corresponding~$D_{i, j}$, in the sense that~$\BD_{i, j} \cap\, T \supset D_{i, j} \cap\, T$. Hence~$T' \subseteq \BD_{1, 2} \cup \BD_{1, 3} \cup \BD_{2, 3}$. Each point of~$X'$ is contained in some bulldozer so there exists a~$\BD_{i', j'}$ whose intersection with~$X'$ is nonempty. Then we keep the corresponding disk~$D_{i', j'}$, which contains at least one point of~$X'$, while removing the other two disks. The desired fourth generalized boundary point is contained in~$D_{i', j'}$. 
    \item If~$R_{T} = 1$, then~$\BD_{i, j} = D_{i, j}$ for all~$1 \leq i < j \leq 3$. Then~$\mathbf{h}$ is the only point of~$T'$ that is not covered by the union~$D_{1, 2} \cup D_{1, 3} \cup D_{2, 3}$. However,~$T$ can be covered by four disks as stated in Lemma~\ref{lem: "unit" triangles have four generalized boundary points}. 
    \item If~$R_{T} < 1$, then the disks~$\BD_{i, j}$ fail to cover some neighborhood of~$\mathbf{h}$. However,~$T$ can be covered by a single disk as stated in Lemma~\ref{lem: small triangles have four generalized boundary points}. 
\end{enumerate}

\begin{figure}[tb]
    \centering
    \begin{subfigure}[b]{0.49\textwidth}
        \centering
        \includegraphics[width=2.3in]{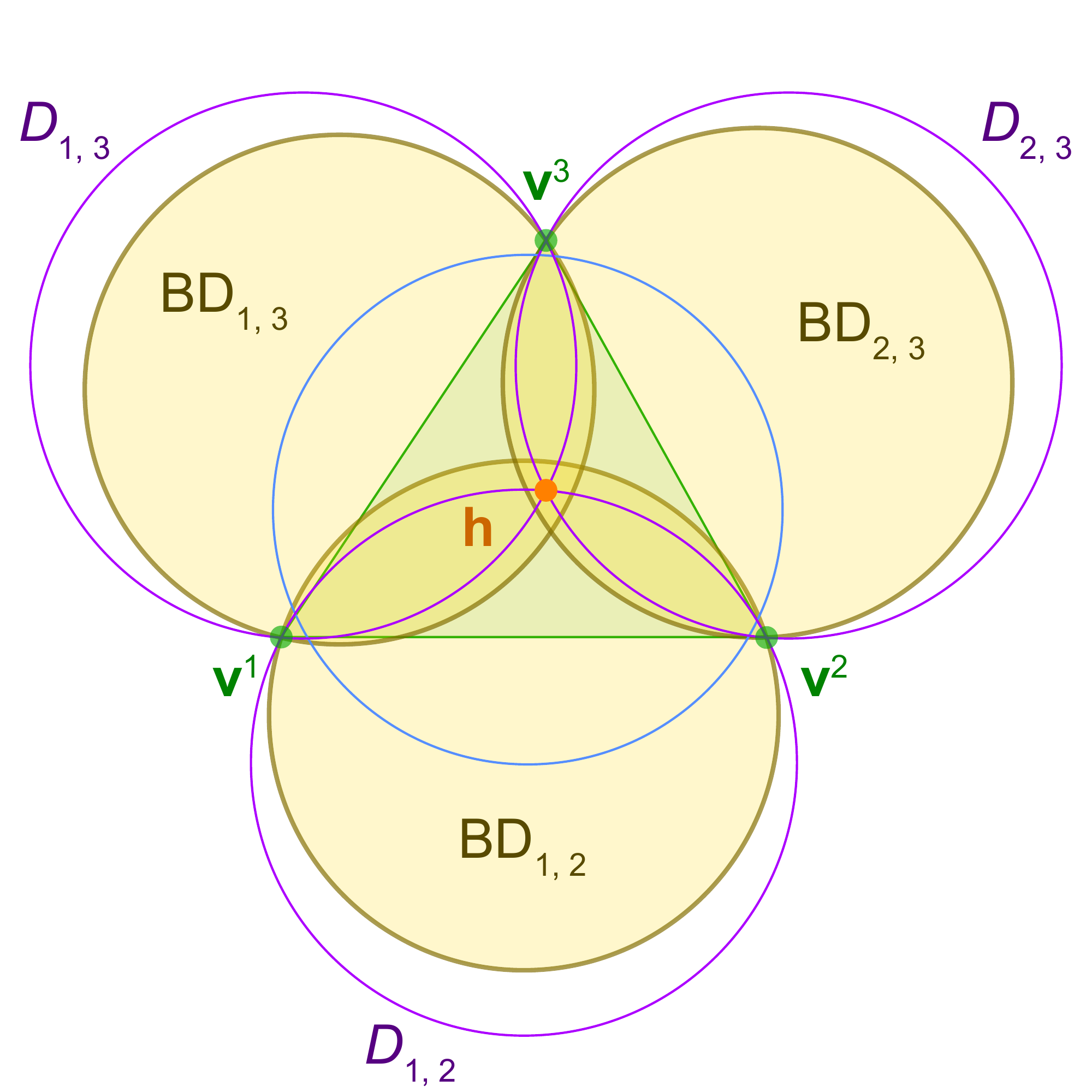}
        \caption{}
        \label{fig: circumcircle 3}
    \end{subfigure}
    \hfill
    \begin{subfigure}[b]{0.49\textwidth}
        \centering
        \includegraphics[width=2.3in]{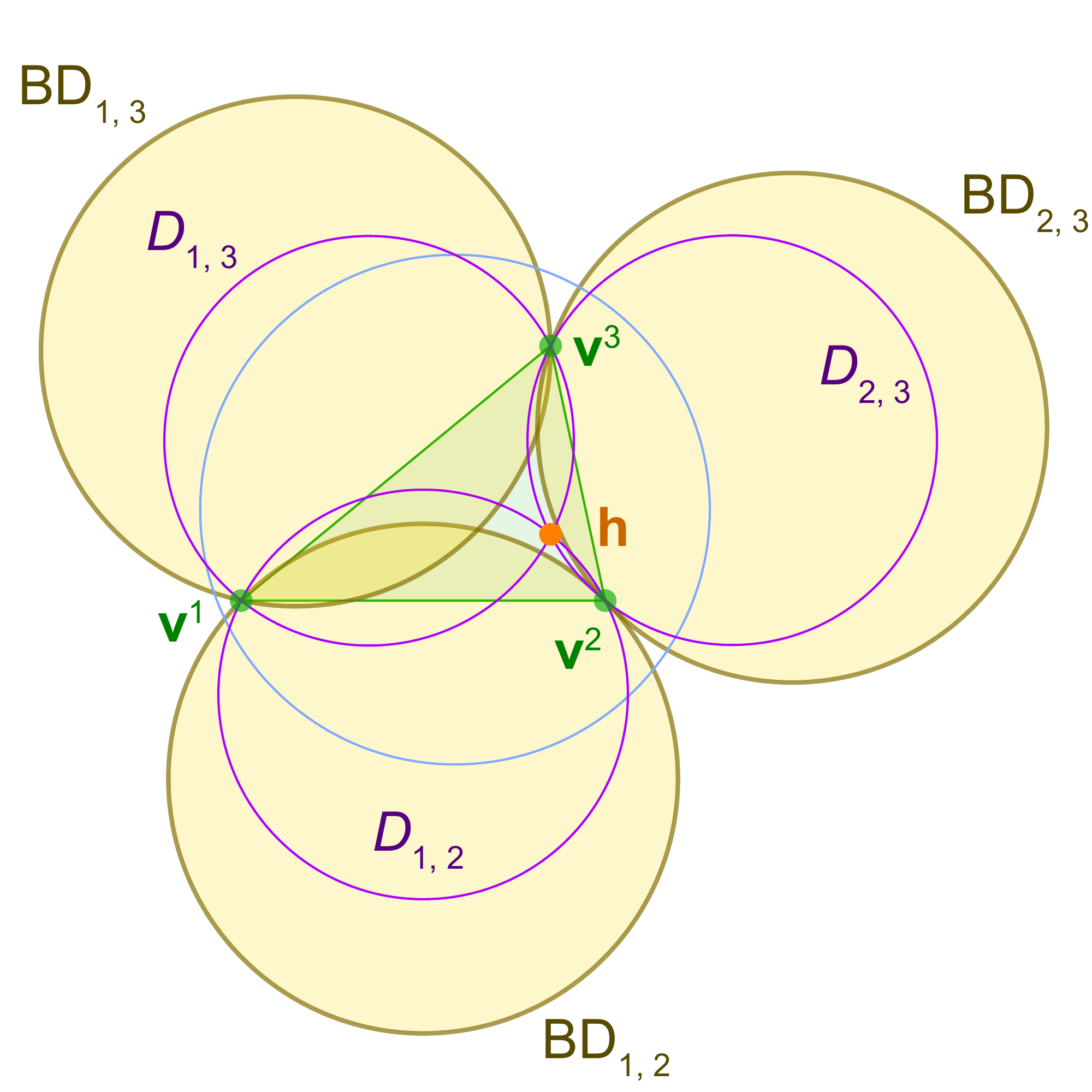}
        \caption{}
        \label{fig: circumcircle 4}
    \end{subfigure}
    \caption{A comparison of~$D_{i, j}$ and~$\BD_{i, j}$. The unit disks~$D_{i, j}$ intersect at the orthocenter~$\mathbf{h}$ and~$T \subseteq D_{1, 2} \cup D_{1, 3} \cup D_{2, 3}$. \\
    \textit{Left:} If~$R_{T} > 1$ (in this case~$\frac{8}{7} \approx 1.071$), then the bulldozers are smaller than the unit disks and hence cover~$T'$ (Lemma~\ref{lem: medium triangles have four generalized boundary points}), but a single unit disk (light blue) cannot cover~$T'$. \\
    \textit{Right:} If~$R_{T} < 1$ (in this case~$\frac{45}{56} \approx 0.804$), then the bulldozers are larger than the unit disks and cannot cover~$T'$, but one unit disk (light blue) covers~$T'$.}
    \label{fig: circumcircle 4 1}
\end{figure}

We require the following technical geometric statements for the case~$R_{T} > 1$. 

\begin{proposition}\label{prop: narrow triangles are small}
    Any right or obtuse triangle~$T$ with longest side of length less than~$2$ is contained inside a unit disk. 
\end{proposition}

\begin{proof}
    See Appendix~\ref{subsec: proofs, boundary points}. 
\end{proof}

\begin{lemma}\label{lem: three bulldozers cover T} 
    Suppose that~$T$ is an acute triangle with all side lengths less than~$2$. If~$R_{T} > 1$ then 
    \[
        T' = T \left\backslash\, \left\{ \mathbf{v}^{1}, \mathbf{v}^{2}, \mathbf{v}^{3}\right\} \right.\! \subseteq \BD_{1, 2} \cup \BD_{1, 3} \cup \BD_{2, 3}.
    \]
\end{lemma}

\begin{proof}
    See Appendix~\ref{subsec: proofs, boundary points}. 
\end{proof}

We now prove Lemma~\ref{lem: medium triangles have four generalized boundary points}, in which~$T$ has short sides but a large circumradius. 

\begin{proof}[Proof of Lemma~\ref{lem: medium triangles have four generalized boundary points}]
    By Proposition~\ref{prop: narrow triangles are small}, we may assume that~$T$ is an acute triangle. By Lemma~\ref{lem: three bulldozers cover T}, 
    \[
        T' \subset \BD_{1, 2} \cup \BD_{1, 3} \cup \BD_{2, 3}. 
    \]
    Pick indices~$i$ and~$j$ such that the interior of~$\BD_{i, j}\left(X\right)$ contains at least one point~$\mathbf{b} \in X'$. Without loss of generality, assume that~$i = 1$,~$j = 2$,~$\mathbf{v}^{1} = \begin{pmatrix} -v_{1} \\ \hphantom{+}0_{\hphantom{1}} \end{pmatrix}$, and~$\mathbf{v}^{2} = \begin{pmatrix} v_{1} \\ 0 \end{pmatrix}$. Let 
    \[
        t_{\max}' \coloneqq \sup\left\{t \geq -1 \,\middle|\, \left(\begin{pmatrix} 0 \\ t \end{pmatrix} + B^{2}\right) \cap X = \emptyset\right\}.
    \]
    Let~$\mathbf{b}$ be any point of~$X'$ on the boundary of~$\begin{pmatrix} 0 \\ t_{\max}' \end{pmatrix} + B^{2}$. This point is the fourth generalized boundary point of~$X$. 
\end{proof}


\subsection{A parameterized version of Inaba's proof}\label{subsec: parametric covering}

Betke, Henk, and Wills (1994)~\cite{BetkeHenkWills1994} introduced the parametric density, a variation of the packing density in which the disks are allowed to overlap, during their work on a packing problem called the Sausage Conjecture~\cite{FejesToth1975}; see~\cite{BetkeHenkWills1994, Boeroeczky2004, HenkWills2020} for further details. Let 
\begin{equation*}
A_{2}\coloneqq\mathbb{Z}\begin{pmatrix} 0 \\ 2 \end{pmatrix}+\mathbb{Z}\begin{pmatrix} \sqrt{3} \\ 1 \end{pmatrix}
\end{equation*} 
be the hexagonal lattice and~$\mathscr{A}_{2}^{\rho} \coloneqq \left\{ D_{\mathbf{c}, \rho} \subset \mathbb{R}^{2} \,\middle|\, \mathbf{c} \in A_{2}\right\}$ be the collection of disks of radius~$\rho \geq 1$ that are centered at the points of~$A_{2}$. This radius~$\rho$ is called the \textit{parameter}; the case~$\rho = 1$ reduces to the usual hexagonal packing in Inaba's proof and we write~$\mathscr{A}_{2} = \mathscr{A}_{2}^{1}$. We call the subset of~$\mathbb{R}^{2}$ covered by exactly one disk~$D_{\mathbf{c}, \rho} \in \mathscr{A}_{2}^{\rho}$ the ``good'' region of~$\mathscr{A}_{2}^{\rho}$ and its complement the ``bad'' region. An exact cover of~$X$ requires each point in~$X$ to avoid the ``bad'' region. If~$\rho > 1$, then neighboring disks of~$\mathscr{A}_{2}^{\rho}$ overlap (Figure~\ref{fig: parametric covering, 1.05, comparison}), so the ``bad'' region includes any part of the plane covered by multiple disks. The critical value for~$\rho$ minimizes the total area of the ``bad'' region and so maximizes the lower bound for~$\widehat{\sigma}_{2}$ (over all coverings of the form~$\mathscr{A}_{2}^{\rho}$). 

\begin{figure}
    \centering
    \includegraphics[width=4in]{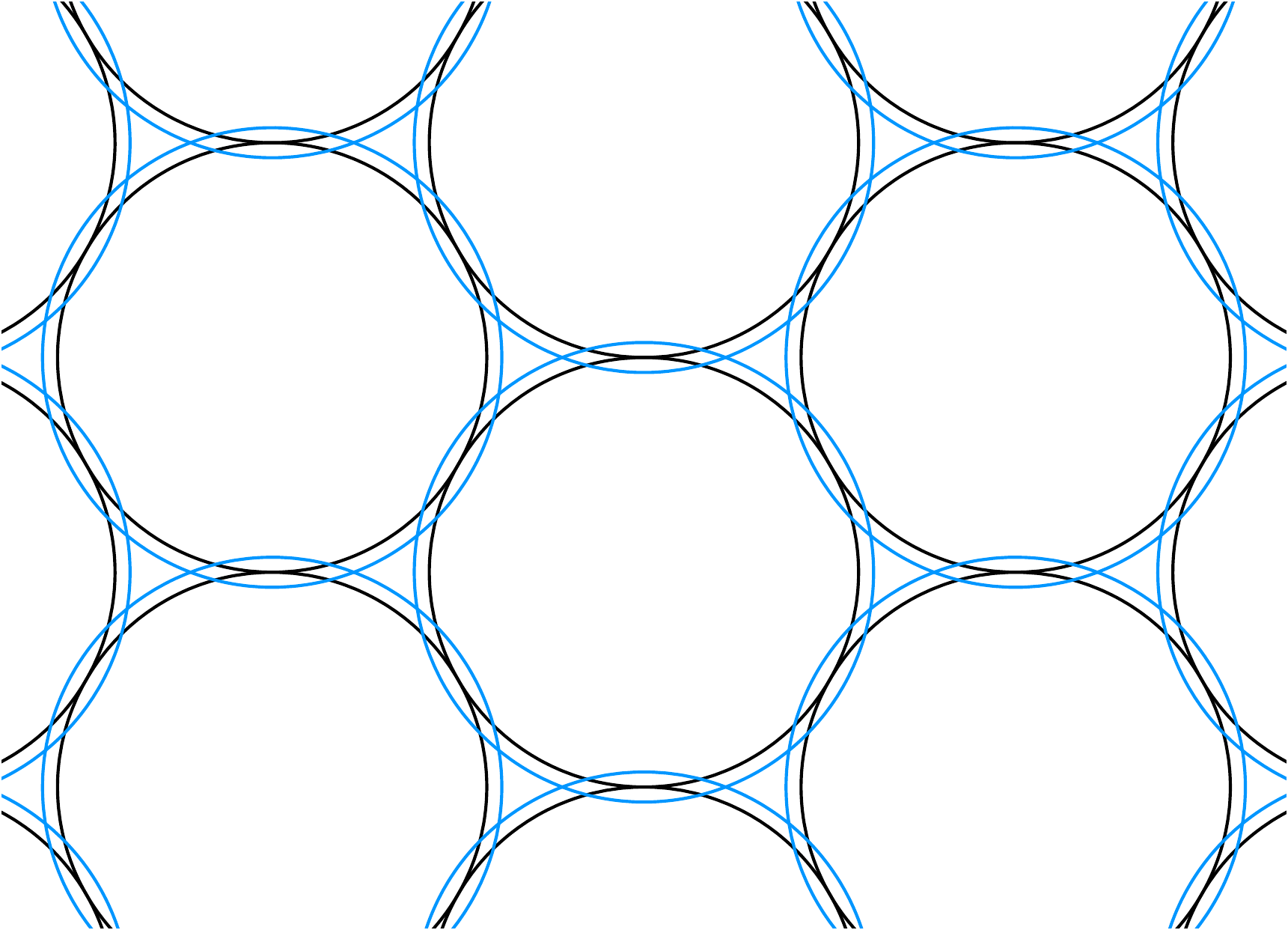}
    \caption{The disks of~$\mathscr{A}_{2}^{\rho}$ for~$\rho = 1$ (black) and~$\rho = 1.07$ (blue).}
    \label{fig: parametric covering, 1.05, comparison}
\end{figure}

We use the same argument as Inaba, but with the parameterized family~$\mathscr{A}_{2}^{\rho}$ of coverings, and combine it with the Extension Argument for another proof of~$\widehat{\sigma}_{2} \geq 16$. However,~$\mathscr{A}_{2}^{\rho}$ has another advantage over~$\mathscr{A}_{2}$. Removing one disk from~$\mathscr{A}_{2}$ strictly expands the ``bad'' region, so is never beneficial for exact covering. However, removing one disk~$D_{\mathbf{c}, \rho}$ from~$\mathscr{A}_{2}^{\rho}$ changes the subsets of~$D_{\mathbf{c}}$ which are covered by another disk from ``bad'' to ``good.'' Our eventual proof of~$\widehat{\sigma}_{2} \geq 17$ in the next subsection will use similar methods to this proof. 

Let~$C$ be a packing set and~$\rho \in \left( 0, \infty \right)$. Our formulations of disk covering have so far been restricted to \textit{unit} disks, but~$\rho B^{2}$ is not a unit disk for~$\rho \neq 1$. However, covering an~$X \in \mathscr{X}$ by~$\left\{ D_{\mathbf{c}, \rho} \subset \mathbb{R}^{2} \,\middle|\, \mathbf{c} \in C\right\}$ is equivalent to covering~$\frac{1}{\rho} X \in \mathscr{X}$ by~$\left\{ D_{\mathbf{c}, 1} \subset \mathbb{R}^{2} \,\middle|\, \mathbf{c} \in \frac{1}{\rho} C\right\}$. Hence an alternate intuition of the the parameterized family~$\mathscr{A}_{2}^{\rho}$ is keeping the disk radius at~$1$ but pushing the disks together by a factor of~$\rho$. 

If~$\rho > 1$, then the disks in~$\mathscr{A}_{2}^{\rho}$ overlap, which presents an additional complication when combining the parameterized extension to Inaba's proof with the generalized boundary point method in the previous subsection. In the proofs of the Extension Argument and its generalization, the (generalized) boundary points of~$X$ were removed from~$X$ before Inaba's probabilistic method was applied. When these (generalized) boundary points were added back to~$X$, each of them was covered by zero or one disk of~$\mathscr{A}_{2}^{\rho}$. However, it is possible for some (generalized) boundary point to be covered by \textit{two} disks of~$\mathscr{A}_{2}^{\rho}$ and neither disk can be removed without breaking the exact cover. To avoid this problem, we make sure that each generalized boundary point avoids these subsets of overlapping disks. 

\begin{definition}\label{def: Overlapping regions, parametric}
    Let~$k \in \mathbb{N}$ and~$\rho > 0$. Define 
    \[
        R_{k}^{\rho}\coloneqq\left\{ \mathbf{z}\in \mathbb{R}^{2}\,\middle|\,\mathbf{z}\text{ is in }k\text{ distinct disks of }\mathscr{A}_{2}^{\rho}\right\}.
    \]
\end{definition}

If~$\mathscr{A}_{2}^{\rho}$ is an exact cover of~$X$, then~$X$ must have the following properties:~$X \subset R_{1}^{\rho}$ and no points of~$X$ may be in~$R_{0}^{\rho}$ or~$R_{k}^{\rho}$ for all~$k \geq 2$. A corresponding result applies to any translate~$\mathbf{t} + \mathscr{A}_{2}^{\rho}$, hence for simplicity we write~$R_{k}^{\rho}$ instead of~$\mathbf{t} + R_{k}^{\rho}$ below. 

We only require~$\rho \in \left[1, \frac{2}{\sqrt{3}}\right]$ for the results in this subsection. While~$\rho$ can be any positive number, Inaba's method applied to~$\mathscr{A}_{2}^{\rho}$ for other values of~$\rho$ provide equal or worse results: 
\begin{enumerate}
    \item If~$\rho < 1$ then the empty space~$R_{0}^{\rho}$ grows, so~$\delta\left(\rho B^{2}, A_{2}\right) < \delta\left(B^{2}, A_{2}\right)$. 
    \item If~$\rho > \frac{2}{\sqrt{3}}$ then~$R_{0}^{\rho} = \emptyset$ and~$R_{1}^{\rho}$ shrinks (and eventually becomes empty), a detrimental outcome compared to~$\rho = \frac{2}{\sqrt{3}}$. 
\end{enumerate}
For~$\rho \in \left[1, \frac{2}{\sqrt{3}}\right]$ we have~$R_{k}^{\rho} = \emptyset$ for all~$k \geq 3$, hence these sets can be ignored. 

Since~$A_{2}$ is a lattice, it suffices to limit our calculations to the Voronoi region of~$A_{2}$; see Section~\ref{sec: upper bound} below for the relevant definition. 

\begin{lemma}\label{lem: parametric lower bound}
    Let~$H \coloneqq \VR_{A_{2}}\left(\mathbf{0}\right)$ be the Voronoi region of~$\mathbf{0} \in A_{2}$. For all~$k \in \mathbb{N}_{0}$ and~$\rho \in \left[1, \frac{2}{\sqrt{3}}\right]$, define 
    \begin{equation}
        f\left(\rho,k\right)\coloneqq\frac{\area\left(H\right)-k\area\left(R_{2}^{\rho}\cap H\right)}{\area\left(R_{0}^{\rho}\cap H\right)+\area\left(R_{2}^{\rho}\cap H\right)}+k. \label{eq: parametric lower bound}
    \end{equation}

    If~$X \in \mathscr{X}'$ has at most~$\left\lceil f\left(\rho, k\right)\right\rceil - 1$ points, including at least~$k$ generalized boundary points, then~$X$ can be exactly covered. 
    (That is,~$\widehat{\sigma}_{2}\left(k\right) \geq \left\lceil f\left(\rho, k\right)\right\rceil - 1$.) 
\end{lemma}

We comment on several special cases and examples of~\eqref{eq: parametric lower bound} below and in~\eqref{eq: 0 boundary points (parametric only)}--\eqref{eq: 4 boundary points}. 
\begin{enumerate}
    \item If~$k = 0$ then 
    \begin{align}
        f\left(\rho,0\right)=\frac{\area\left(H\right)}{\area\left(R_{0}^{\rho}\cap H\right)+\area\left(R_{2}^{\rho}\cap H\right)}. \label{eq: f(rho, 0)}
    \end{align}
    This case is equivalent to not using the Generalized Extension Argument as we do not treat the generalized boundary points differently from any other point of~$X$. 
    \item If~$\rho = 1$ then 
    \begin{eqnarray}
        f\left(1,k\right) & = & \frac{\area\left(H\right)}{\area\left(R_{0}^{1}\cap H\right)}+k \label{eq: f(1, k)} \\
         & = & \frac{1}{1-\delta\left(B^{2},A^{2}\right)}+k \nonumber \\
         & = & 10.741...+k. \nonumber
    \end{eqnarray}
    This case is equivalent to only using (Part 1 of) the Generalized Extension Argument as it ignores the parameterized family~$\mathscr{A}_{2}^{\rho}$ for all~$\rho > 1$. 
    \item If~$\rho = 1$ and~$k = 0$ then we recover Inaba's method: 
    \begin{align}
        f\left(1, 0\right) = 10.741...\,. \label{eq: f(1, 0)}
    \end{align}
\end{enumerate}

\begin{proof}
    Fix~$\rho \in \left[1, \frac{2}{\sqrt{3}}\right]$. Suppose that~$X \in \mathscr{X}'$ has~$k$ generalized boundary points. We wish to find a translation of~$\mathscr{A}_{2}^{\rho}$ that meets the following conditions: 
    \begin{enumerate}
        \item Each of the~$k$ generalized boundary points~$\mathbf{b} \in X$ is not in~$R_{2}^{\rho}$, so must be in~$R_{0}^{\rho}$ or~$R_{1}^{\rho}$. 
        
        Let~$\mathscr{D}''$ be a disjoint cover of~$X''$. We cover the generalized boundary points in~$R_{0}^{\rho}$ with a collection of (possibly additional) disks as in the proof of the Generalized Extension Argument~\ref{lem: Kozma method, improved}. On the other hand, each generalized boundary point in~$R_{1}^{\rho}$ is already covered by a disk in~$\mathscr{D}''$. 
        \item Each of the~$n - k$ remaining points~$\mathbf{x} \in X$ is not in~$R_{0}^{\rho}$ or~$R_{2}^{\rho}$, so must be in~$R_{1}^{\rho}$. 
    \end{enumerate}
    The sets of forbidden translation factors---see~\eqref{eq: forbidden set}---of~$\mathbf{b}$ and~$\mathbf{x}$ as stated above are~$-\mathbf{b} + R_{2}^{\rho}$ and~$-\mathbf{x} + \left(R_{2}^{\rho} \cup R_{2}^{\rho}\right)$ respectively. 
    
    If~$n$ satisfies 
    \[
        \area\left(H\right) > k \area\left(R_{2}^{\rho} \cap H\right) + \left(n - k\right) \left(\area\left(R_{0}^{\rho} \cap H\right) + \area\left(R_{2}^{\rho} \cap H\right)\right), 
    \]
    then~$\mathscr{A}_{2}^{\rho}$ can be translated to meet conditions 1 and 2. Hence there is some exact cover~$\mathscr{D} \supseteq \mathscr{A}_{2}^{\rho}$ of~$X$. Solving the above inequality for~$n$ yields
    \[
        n < \frac{\area\left(H\right)-k\area\left(R_{2}^{\rho}\cap H\right)}{\area\left(R_{0}^{\rho}\cap H\right)+\area\left(R_{2}^{\rho}\cap H\right)}+k.
    \]
    So~$X$ can be exactly covered if~$\left\vert X \right\vert \leq \left\lceil f\left(\rho, k\right)\right\rceil - 1$, cf.~\eqref{eq: general lower bound}. 
\end{proof}

\begin{proof}[Alternate proof of~$\widehat{\sigma}_{2} \geq 16$]
    Let~$X \in \mathscr{X}'$. We use elementary geometry to calculate the areas mentioned in~\eqref{eq: parametric lower bound}: 
    \begin{alignat*}{1}
        \area\left(H\right) & =2\sqrt{3},\\
        \area\left(R_{2}^{\rho}\cap H\right) & =6\left(\rho^{2}\arcsec\left(\rho\right)-\sqrt{\rho^{2}-1}\right),\\
        \area\left(R_{0}^{\rho}\cap H\right) & =2\sqrt{3}-\pi\rho^{2}+\area\left(R_{2}^{\rho}\right).
    \end{alignat*}
    Since~$\rho \in \left[1, \frac{2}{\sqrt{3}}\right]$ was picked arbitrarily, we maximize~$f\left(\rho, k\right)$ for each~$k \in \mathbb{N}_{0}$. Let
    \[
        f_{\max}\left(k\right)\coloneqq\max_{\rho\in\left[1,\frac{2}{\sqrt{3}}\right]}f\left(\rho,k\right)
        \qquad\text{and}\qquad 
        \rho_{\max}\left(k\right)\coloneqq\argmax_{\rho\in\left[1,\frac{2}{\sqrt{3}}\right]}f\left(\rho,k\right)\text{.}
    \]
    By inspection (see Figure~\ref{fig: f}) and~\eqref{eq: f(rho, 0)}--\eqref{eq: f(1, 0)}, we obtain 
    \begin{align}
        f\left(1,0\right) & =10.741...\,, & f_{\max}\left(0\right) & =13.928...\quad\text{at}\quad\rho_{\max}\left(0\right)=1.035...\,\text{,}\label{eq: 0 boundary points (parametric only)}\\
        f\left(1,3\right) & =13.741...\,, & f_{\max}\left(3\right) & =16.152...\quad\text{at}\quad\rho_{\max}\left(3\right)=1.028...\,\text{,}\label{eq: 3 boundary points}\\
        f\left(1,4\right) & =14.741...\,, & f_{\max}\left(4\right) & =16.948...\quad\text{at}\quad\rho_{\max}\left(4\right)=1.026...\,\text{,}\label{eq: 4 boundary points}\\
        f\left(1,5\right) & =15.741...\,, & f_{\max}\left(5\right) & =17.766...\quad\text{at}\quad\rho_{\max}\left(5\right)=1.024...\,\text{.}\label{eq: 5 boundary points}
    \end{align}
    
    The lower bound of~$\widehat{\sigma}_{2} \geq 16$ follows from~\eqref{eq: 4 boundary points}, Lemma~\ref{lem: parametric lower bound}, and the Generalized Extension Argument~\ref{lem: Kozma method, improved}. 
\end{proof}

\begin{figure}
    \centering
    \includegraphics[width=4.5in]{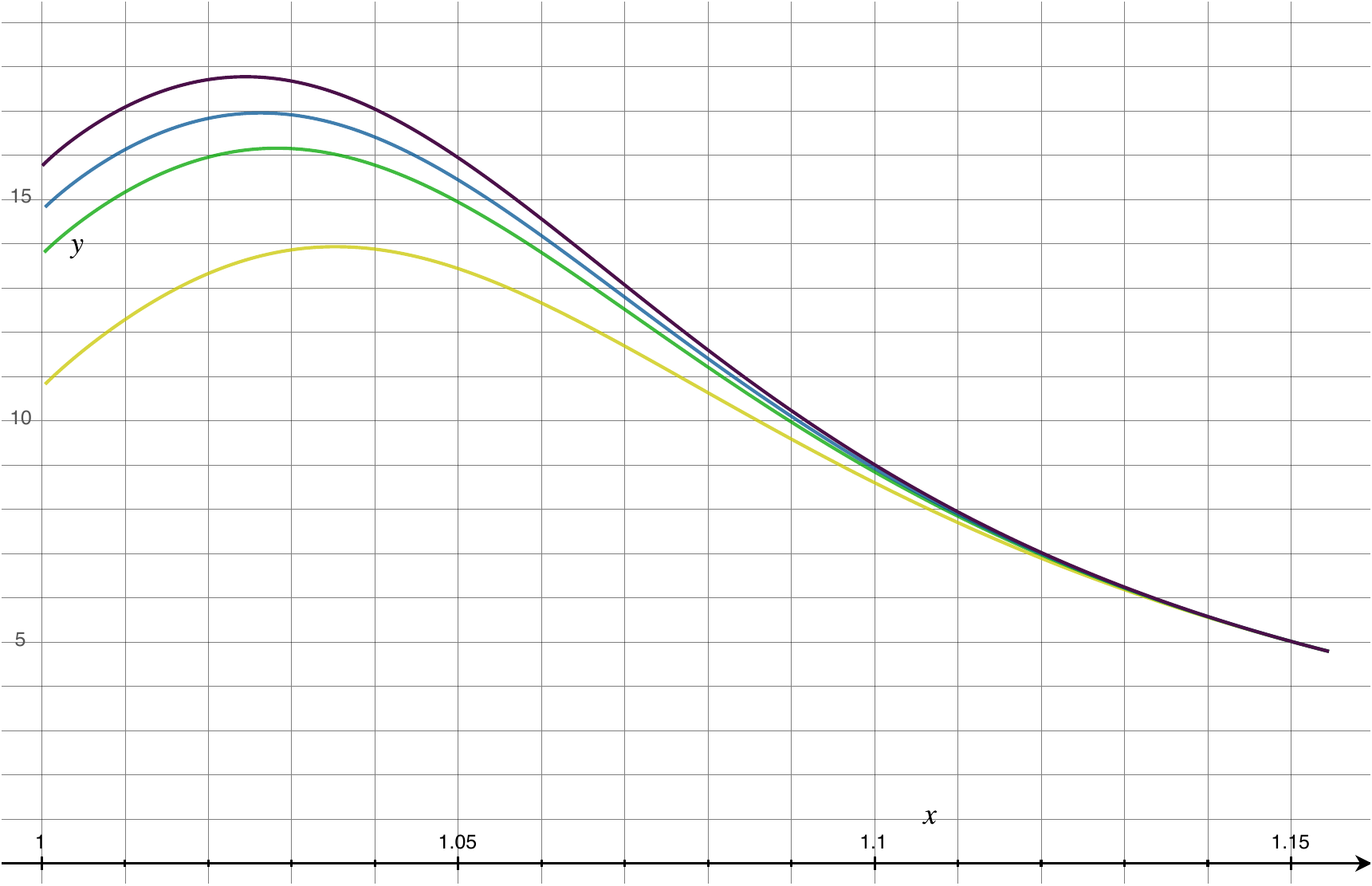}
    \caption{The graphs of~$y = f\left(\rho, k\right)$ for~$\rho = x \in \left[1, \frac{2}{\sqrt{3}}\right]$ ($x$-axis) and~$k = 0$ (light yellow),~$k = 3$ (green),~$k = 4$ (teal), and~$k = 5$ (dark purple).}
    \label{fig: f}
\end{figure}

\begin{remark}
    Unlike the proof of~$\widehat{\sigma}_{2} \geq 16$ in the previous section, we did not use Aloupis, Hearn, Iwasawa, and Uehara's result of~$\sigma_{2} \geq 12$ in this proof. Their proof used specific technical arguments on subsets of vertical lines in~$\mathbb{R}^{2}$, which we were unable to beneficially adapt to the parametric context. 
\end{remark}

The benefit of the parameterized family~$\mathscr{A}_{2}^{\rho}$ compared to the usual ``nonparametric'' hexagonal lattice~$\mathscr{A}_{2}$ is cut-and-dry: The~$\rho_{\max} > 1$ case is always better than or equal to the~$\rho = 1$ case. For each~$k \in \mathbb{N}_{0}$, we have~$f_{\max}\left(k\right) \geq f\left(1, k\right)$ by definition of~$f_{\max}$. However, as a consequence of~$\left.\frac{\partial}{\partial \rho} f\left(\rho, k\right)\right|_{\rho = 1} > 0$, we also have~$f_{\max}\left(k\right) > f\left(1, k\right)$. This inequality and its counterpart~$\left.\frac{\partial}{\partial \rho} f\left(\rho, k\right)\right|_{\rho = \frac{2}{\sqrt{3}}} < 0$ are intuitively justified by comparing~$\area\left(R_{0}^{\rho}\right)$ and~$\area\left(R_{2}^{\rho}\right)$ for~$\rho = 1$ and~$1 + \varepsilon$, and~$\rho = \frac{2}{\sqrt{3}}$ and~$\frac{2}{\sqrt{3}} - \varepsilon$ respectively. 

On the other hand, the function~$f$ is not strong enough to prove the result~$\widehat{\sigma}_{2} \geq 12 + k$ for all~$k \in \mathbb{N}$. Below we compare the two methods that obtain~$\widehat{\sigma}_{2} \geq 16$: 
\begin{enumerate}
    \item The combination of Aloupis, Hearn, Iwasawa, and Uehara's~\cite{AloupisHearnIwasawaUehara2012} lower bound and Part 1 of the Generalized Extension Argument~\ref{lem: Kozma method, improved}: 
    \begin{equation}
        \widehat{\sigma}_{2}\left(k\right) \geq 12 + k \label{eq: lower bound: AHIU + Kozma}
    \end{equation}
    \item The combination of Inaba's original lower bound (Theorem~\ref{thm:inaba}), Part 1 of the Generalized Extension Argument~\ref{lem: Kozma method, improved}, and the parameterized version of Inaba's proof (Lemma~\ref{lem: parametric lower bound}): 
    \begin{equation}
        \widehat{\sigma}_{2}\left(k\right) \geq \left\lceil f_{\max}\left(k\right) \right\rceil - 1 = 
        \begin{cases}
            13 + k & \hphantom{0}0\leq k\leq3,\\
            12 + k & \hphantom{0}4\leq k\leq11,\\
            11 + k & 12\leq k\leq45,\\
            10 + k & 46\leq k.
        \end{cases} 
        \label{eq: lower bound: Inaba + Kozma + parametric}
    \end{equation}
\end{enumerate}

As shown\footnote{
    An explicit formulation of the local maximum point~$\rho_{\max}\left(k\right)$ seems to be difficult. The equation~$\frac{\partial}{\partial \rho} f\left(\rho, k\right) = 0$ which~$\rho_{\max}\left(k\right)$ satisfies is equivalent (for our purposes) to
    \begin{equation}
        6\left(k+2\right)\arcsec\left(\rho\right)-\sqrt{3}\pi k\sqrt{\rho^{2}-1}-\pi=0. \label{eq: solve this equation to find rho_max}
    \end{equation}
    Note that~$\arcsec\left(\rho\right) = \sqrt{\rho^{2} - 1}$ at~$\rho = 1$ and~$\frac{d}{d \rho}\left(\sqrt{\rho^{2} - 1} - \arcsec\left(\rho\right)\right) = \sqrt{1 - \frac{1}{\rho^{2}}}$ is small for~$\rho \approx 1$. 
    We substitute~$\arcsec\left(\rho\right)$ for~$\sqrt{\rho^{2} - 1}$ in~\eqref{eq: solve this equation to find rho_max} so it becomes a linear polynomial in~$\arcsec\left(\rho\right)$. The relevant solution is~$\rho = \widetilde{\rho}_{\max}\left(k\right) \coloneqq \sec \left(\frac{\pi}{6 k + 12 - \sqrt{3} \pi k}\right)$, which is near~$\rho_{\max}\left(k\right)$. In Figure~\ref{fig: AHIU vs. parametric covering}, we use~$f\left(\widetilde{\rho}_{\max}\left(k\right), k\right)$ instead of~$f_{\max}\left(k\right)$. 
} 
in Figure~\ref{fig: AHIU vs. parametric covering}, for~$k \geq 12$, the Generalized Extension Argument obtains a better bound for~$\widehat{\sigma}_{2}$ than Lemma~\ref{lem: parametric lower bound}, while for~$k \leq 3$ the opposite holds. However, by Part 2 of the Generalized Extension Argument, we may assume that~$k \geq 4$. Hence the first result is actually stronger than the second in practice. Eventually the parameter does not deliver any benefit---Inaba with the generalized boundary points already achieves~$\widehat{\sigma}_{2}\left(k\right) \geq 10 + k$, as shown in~\eqref{eq: f(1, 0)}. 

\begin{figure}
    \centering
    \includegraphics[width=4.5in]{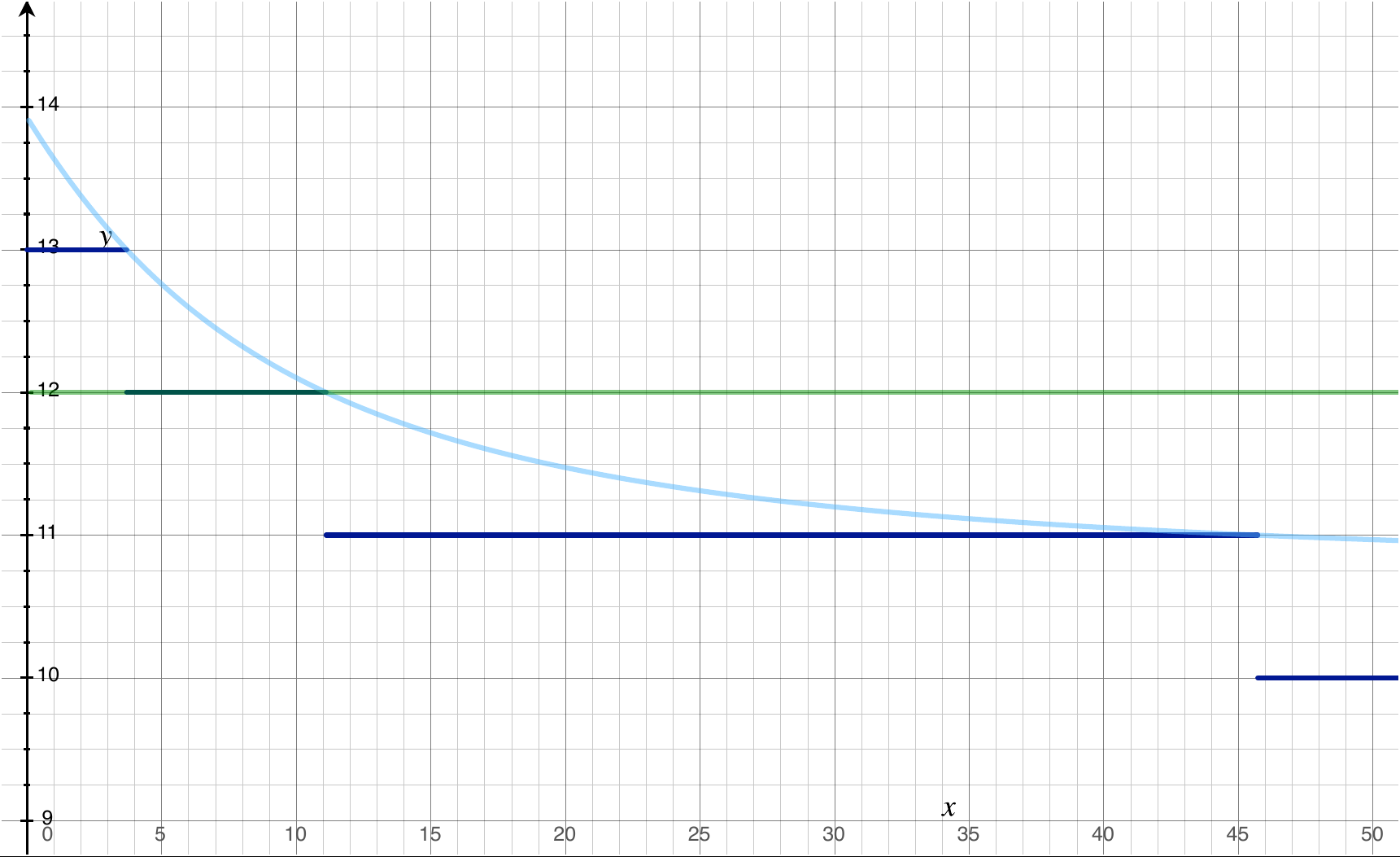}
    \caption{A comparison of close approximations to~$y = f_{\max}\left(k\right) - k$ (light blue) and~$y = \left(\left\lceil f_{\max}\left(k\right) \right\rceil - 1\right) - k$ (dark blue) with~$y = \left(12 + k\right) - k$ (green), where~$k = x$.} 
    \label{fig: AHIU vs. parametric covering}
\end{figure}

\begin{remark}
    We conjecture that the improvement that~$f_{\max}\left(k\right)$ brings to the table over~$f_{\max}\left(0\right)$ is strictly ``subadditive'' in the following sense. Let~$g\left(\rho, k\right) \coloneqq f\left(\rho, k\right) - f\left(\rho, 0\right)$ and~$g_{\max}\left(k\right) \coloneqq f_{\max}\left(k\right) - f_{\max}\left(0\right)$. For all~$k_{1}, k_{2} \geq 0$, we have (with the corresponding equalities for~$g$ and~$f$ listed as a comparison) 
    \begin{align*}
        g\left(1,k_{1}+k_{2}\right) & =g\left(1,k_{1}\right)+g\left(1,k_{2}\right), \\
        g_{\max}\left(k_{1}+k_{2}\right) & <g_{\max}\left(k_{1}\right)+g_{\max}\left(k_{2}\right),
    \end{align*}
    and alternatively 
    \begin{align}
        f\left(1,k_{1}+k_{2}\right) & =f\left(1,k_{1}\right)+k_{2}, \nonumber \\
        f_{\max}\left(k_{1}+k_{2}\right) & <f_{\max}\left(k_{1}\right)+k_{2}. \label{eq: subadditivity of f}
    \end{align}
    If we show that \textit{every}~$X \in \mathscr{X}'$ has at least~$4 + k$ generalized boundary points for some~$k > 0$, then~$\widehat{\sigma}_{2} \geq \left\lceil f_{\max}\left(4 + k\right) \right\rceil - 1$ by Lemma~\ref{lem: parametric lower bound}. However, the subadditivity of~\eqref{eq: subadditivity of f} implies that adding the~$k$ additional boundary points will improve the lower bound for~$\widehat{\sigma}_{2}$ by \textit{less than} or equal to~$k$. The reason is the overlapping disks when~$\rho > 1$. When we performed our calculations with~$\mathscr{A}_{2}^{\rho}$ in the proof of Lemma~\ref{lem: parametric lower bound}, we had to exclude the (generalized) boundary points from~$R_{2}^{\rho}$, so they could no longer be added back ``for free.'' 
\end{remark}


\subsection{A redundant disk}\label{subsec: redundant disk}

Suppose that~$X \in \mathscr{X}$ has a triangular or quadrilateral convex hull,~$\mathbf{v}^{1} \in X$ is a boundary point that is covered by two disks of~$\mathscr{A}_{2}^{\rho}$, and~$\mathscr{A}_{2}^{\rho}$ is an exact cover of~$X \left\backslash\, \left\{\mathbf{v}^{1}\right\}\right.$. Under certain conditions, we can remove one of the disks~$D$ that covers~$\mathbf{v}^{1}$ without breaking the exact cover. In other words, although~$\mathscr{A}_{2}^{\rho}$ is not an exact cover of~$X$, we show that~$\mathscr{A}_{2}^{\rho} \left\backslash\, \left\{D\right\} \right.$ is an exact cover of~$X$. Hence we may call~$D$ a ``redundant disk.'' This method offers a slight benefit: 

\begin{definition}
    Define the following subsets of~$\mathbb{R}^{2}$ as follows (see Figure~\ref{fig: redundant disk 1}): 
    \begin{enumerate}
        \item Let~$R_{2, 1}^{\rho}$ be the set of points in~$\mathbb{R}^{2}$ which are contained in the intersection of two disks centered at different~$x$-coordinates. 
        \item Let~$R_{2, 2}^{\rho}$ be the set of points in~$\mathbb{R}^{2}$ which are contained in the intersection of two disks centered at the same~$x$-coordinate. 
        \item For any distinct~$\mathbf{y}, \mathbf{z} \in A_{2}$, let~$R^{\rho}\left(\mathbf{y}, \mathbf{z}\right) \coloneqq \rho D_{\mathbf{y}} \cap \rho D_{\mathbf{z}}$. 
    \end{enumerate}
\end{definition}

\begin{figure}[tb]
    \centering
    \begin{subfigure}[b]{0.49\textwidth}
        \centering
        \includegraphics[width=2.3in]{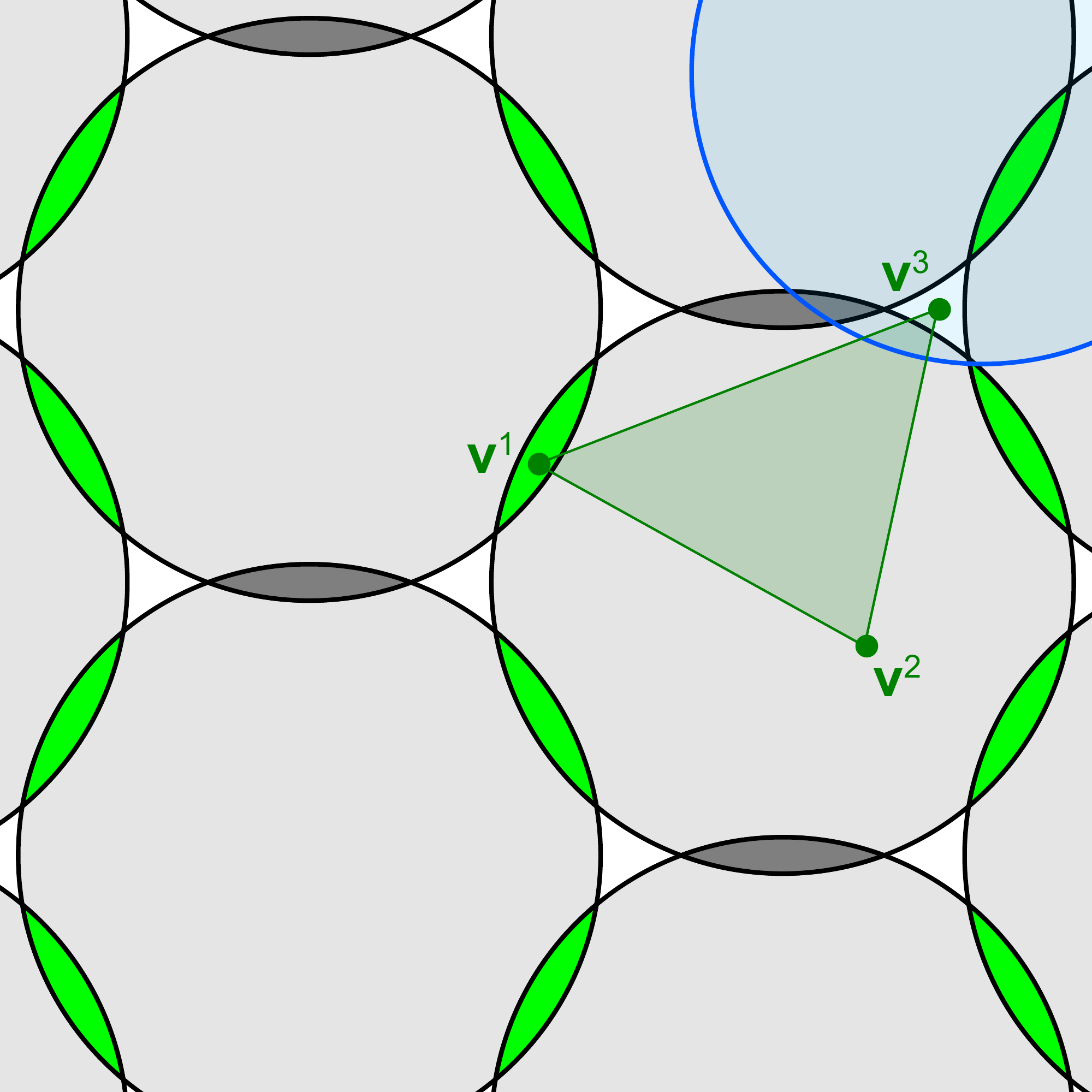}
        \caption{}
        \label{fig: redundant disk 1}
    \end{subfigure}
    \hfill
    \begin{subfigure}[b]{0.49\textwidth}
        \centering
        \includegraphics[width=2.3in]{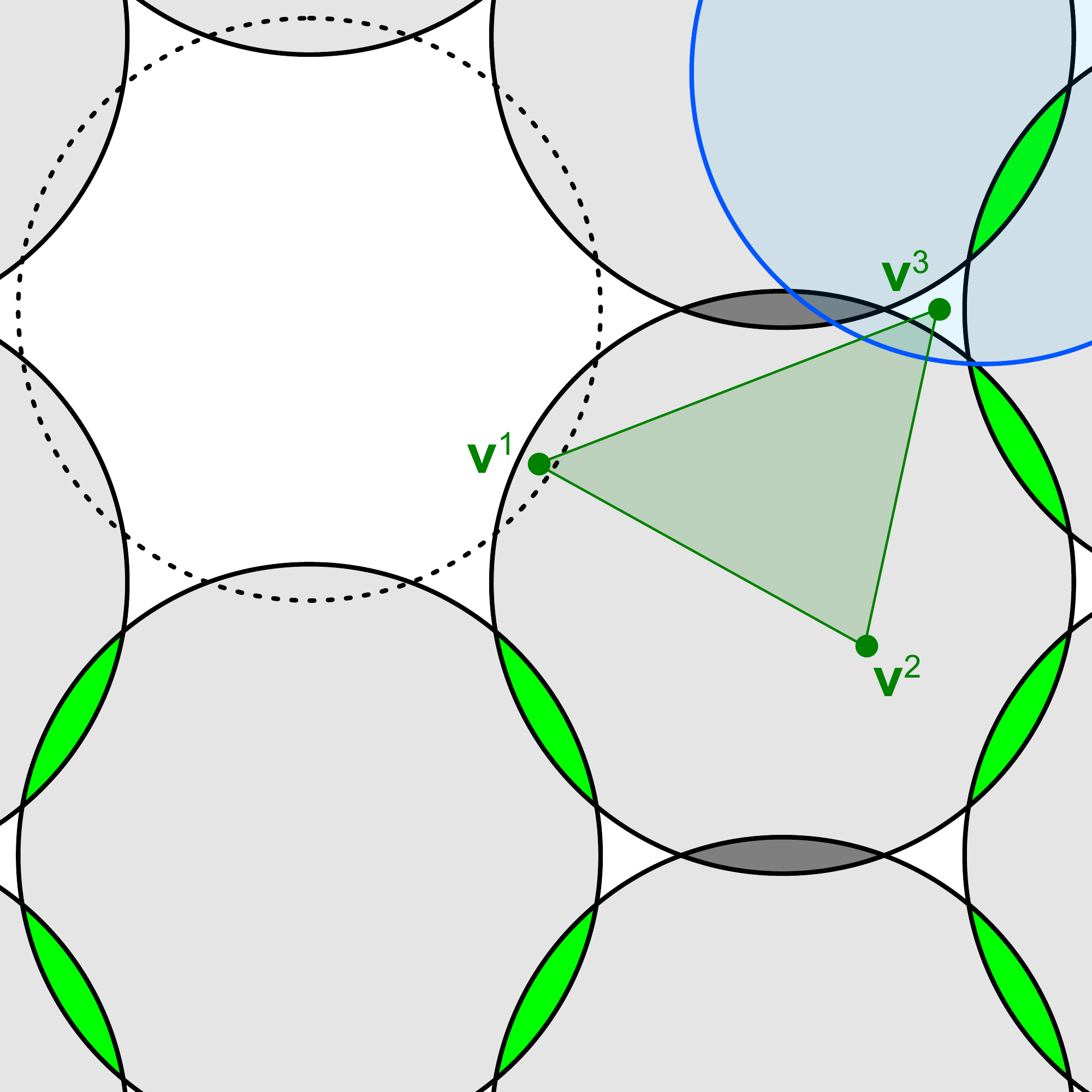}
        \caption{}
        \label{fig: redundant disk 2}
    \end{subfigure}
    \caption{A demonstration of the redundant disk method (Proposition~\ref{prop: a diagonal lens gives a redundant disk}). The collection~$\mathscr{A}_{2}^{\rho}$ (light grey disks) for~$\rho = \frac{8}{7} \approx 1.067$; the subsets~$R_{2, 1}^{\rho}$ and~$R_{2, 2}^{\rho}$ are light green and dark grey respectively. The set~$T'$ is indicated by the dark green triangle. The blue disk is an extra disk necessary for the exact covering of~$X$. \\
    \textit{Left:} The set~$R_{2, 1}^{\rho}$ contains a vertex with small angle less than~$\frac{\pi}{2}$. \\
    \textit{Right:} A disk can be removed from~$\mathscr{A}_{2}^{\rho}$ without breaking the exact cover (see Proposition~\ref{prop: a diagonal lens gives a redundant disk}).}
    \label{fig: redundant disk 1 2}
\end{figure}

Note that~$R_{2}^{\rho} = \bigcup_{\mathbf{y} \neq \mathbf{z}} R^{\rho}\left(\mathbf{y}, \mathbf{z}\right)$ and each point of~$R_{2}^{\rho}$ is contained in either~$R_{2, 1}^{\rho}$ or~$R_{2, 2}^{\rho}$.

Suppose that~$X$ is exactly covered by~$\mathscr{A}_{2}^{\rho}$ (or any covering of radius-$\rho$ disks that contains~$\mathscr{A}_{2}^{\rho}$). Without loss of generality, assume that~$\mathbf{v}^{1} \in R^{\rho}\left(\mathbf{y}, \mathbf{z}\right) \subset R_{2, 1}^{\rho}$,~$\mathbf{y}$ is to the left of~$\mathbf{z}$, and~$X \left\backslash\, \left\{\mathbf{v}^{1}\right\} \right.$ is to the right of~$\mathbf{y}$, as shown in Figure~\ref{fig: redundant disk 1}. If~$\conv X$ is ``pointy'' enough at~$\mathbf{v}^{1}$, then the ``roundness'' of the disk~$D_{\mathbf{z}}$ ensures that~$D_{\mathbf{z}}$ covers all nearby areas of~$\conv X$ without any help from additional disks. In particular,~$D_{\mathbf{y}}$ can be removed from the packing~$\mathscr{A}_{2}^{\rho}$ without breaking the exact cover, even though~$D_{\mathbf{y}} \cap \conv X \neq \emptyset$ (Figure~\ref{fig: redundant disk 2}). Then the subset~$R_{2, 1}^{\rho}$ can be ignored and treated as part of~$R_{1}^{\rho}$ from the perspective of~$\mathbf{v}^{1}$, resulting in a slightly better (or equal) lower bound for~$\widehat{\sigma}_{2}$ (Lemma~\ref{lem: f_r}). 

We carefully analyze the geometry of~$R_{2, 1}^{\rho}$ and~$R_{2, 2}^{\rho}$ to formalize the above intuition. 

\begin{remark}
    Figures~\ref{fig: redundant disk 3} and~\ref{fig: redundant disk 4} show that unlike~$R_{2, 1}^{\rho}$, the set~$R_{2, 2}^{\rho}$ lacks redundancy. If~$\mathbf{v}^{1} \in R_{2, 2}^{\rho}$ then in general, no disk that covers~$X$ can be removed without breaking the exact cover. 

    \begin{figure}[tb]
        \centering
        \begin{subfigure}[b]{0.49\textwidth}
            \centering
            \includegraphics[width=2.3in]{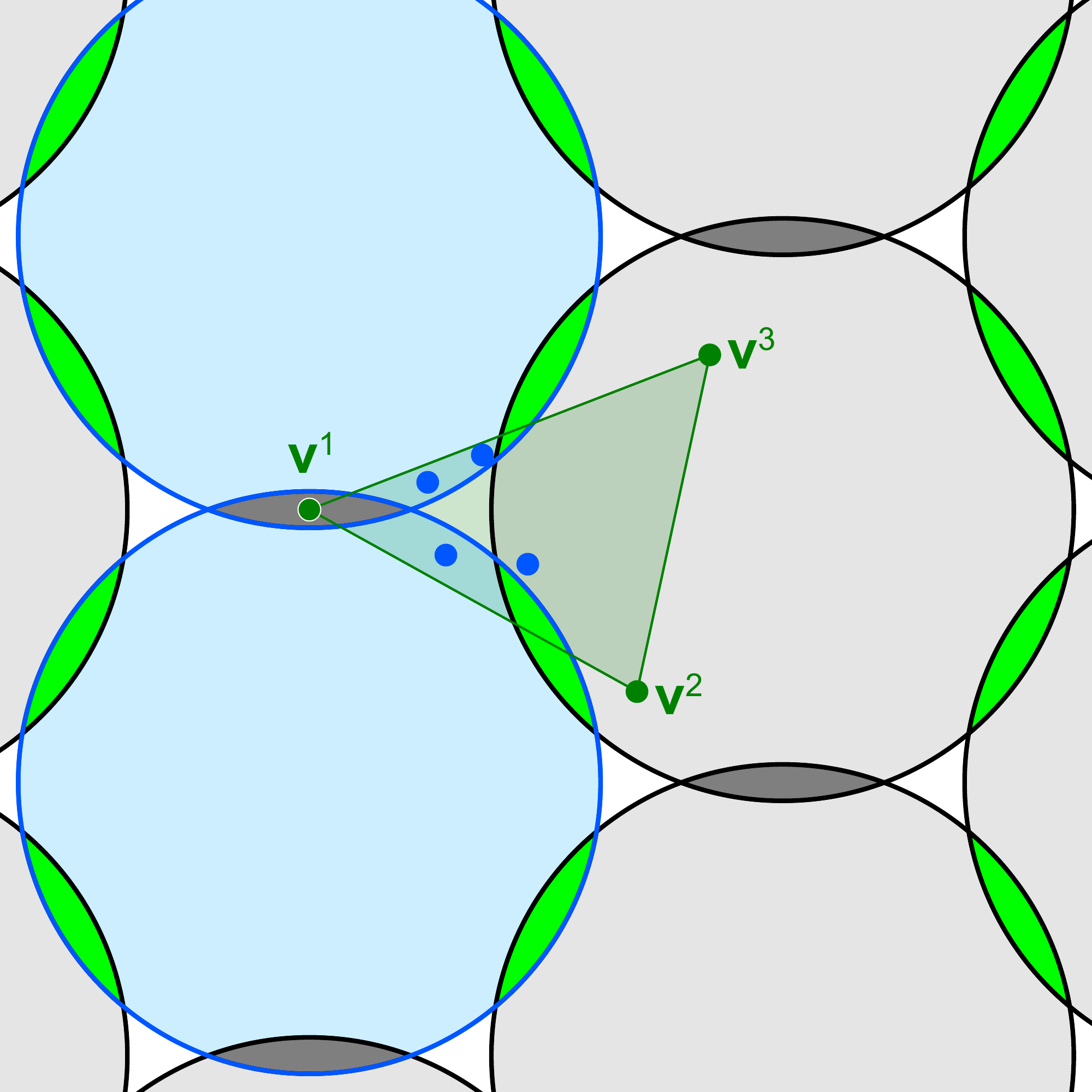}
            \caption{}
            \label{fig: redundant disk 3}
        \end{subfigure}
        \hfill
        \begin{subfigure}[b]{0.49\textwidth}
            \centering
            \includegraphics[width=2.3in]{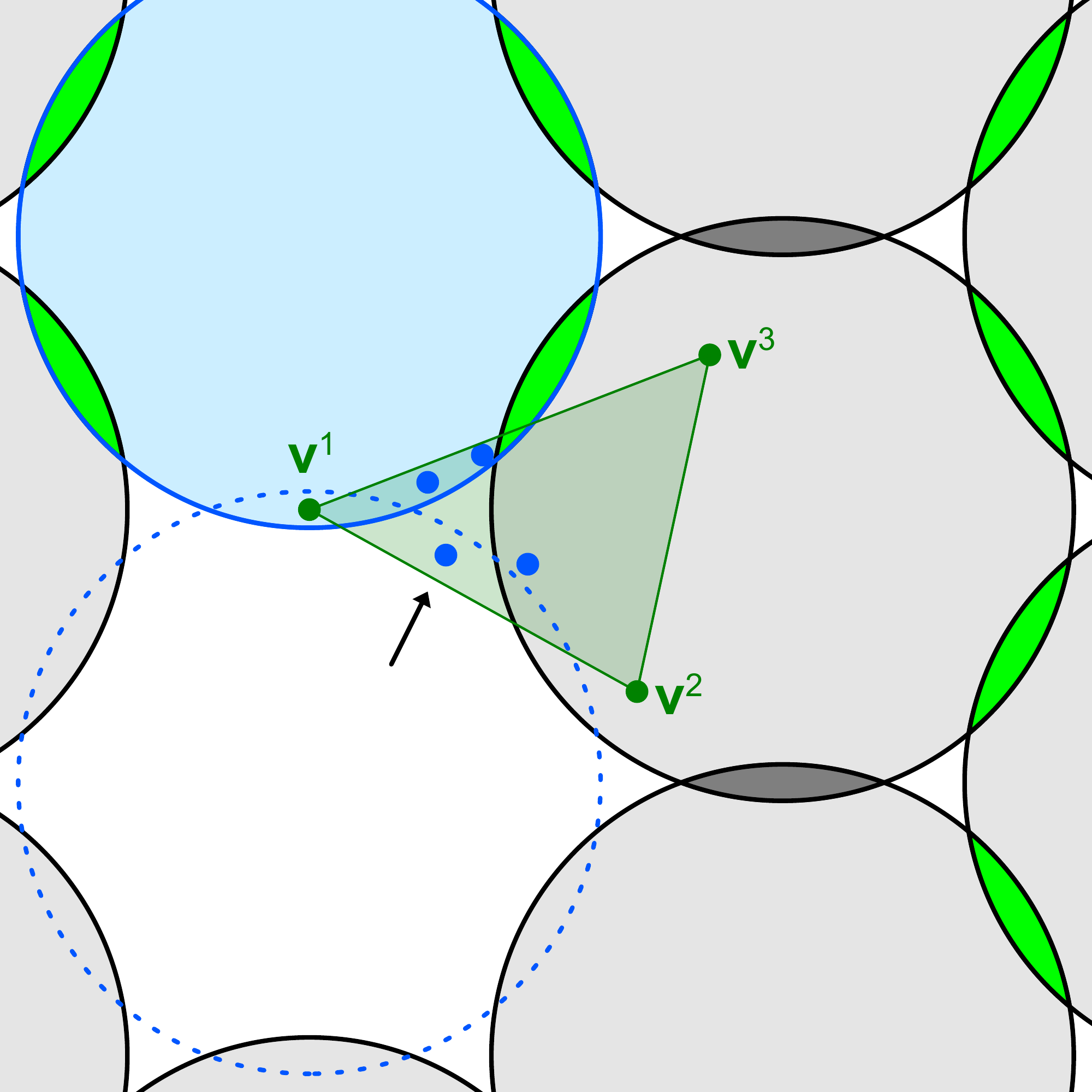}
            \caption{}
            \label{fig: redundant disk 4}
        \end{subfigure}
        \caption{If~$\mathbf{v}^{1} \in R_{2, 2}^{\rho}$, then a disk cannot necessarily be removed from~$\mathscr{A}_{2}^{\rho}$ without breaking the exact covering, as shown by the point in Figure~\ref{fig: redundant disk 4} marked by an arrow.}
        \label{fig: redundant disk 3 4}
    \end{figure}

\end{remark}

\begin{proposition}\label{prop: diagonal lenses are steep}
    Let~$\rho \in \left[1, \sqrt{6} - \sqrt{2} \approx 1.0353\right]$. Let~$\mathbf{y}, \mathbf{z} \in A_{2}$ with~$y_{1} < z_{1}$, and~$\mathbf{w} \in \bd\left(R^{\rho}\left(\mathbf{y}, \mathbf{z}\right)\right)$.
    The slope of the tangent line~$L^{\rho}\left(\mathbf{w}\right)$ to~$R^{\rho}\left(\mathbf{y}, \mathbf{z}\right)$ at~$\mathbf{w}$ has absolute value greater than or equal to~$1$. 
\end{proposition}

\begin{proof}
    See Appendix~\ref{subsec: proofs, redundant disk}. 
\end{proof}

\begin{proposition}\label{prop: a diagonal lens gives a redundant disk}
    Let~$X \in \mathscr{X}$ and~$\rho \in \left[1, \sqrt{6} - \sqrt{2}\right]$. Suppose that~$X$ meets the following conditions (Figure~\ref{fig: redundant disk 1}): 
    \begin{enumerate}
        \item There exist~$\mathbf{v}^{2}, \mathbf{v}^{1}, \mathbf{v}^{3} \in X$ such that the interior angle~$\angle \mathbf{v}^{2} \mathbf{v}^{1} \mathbf{v}^{3} \leq \frac{\pi}{2}$. 
        \item The slopes of~$\overline{\mathbf{v}^{1} \mathbf{v}^{2}}$ and~$\overline{\mathbf{v}^{1} \mathbf{v}^{3}}$ have equal magnitude and opposite signs. 
        \item Every point of~$X \left\backslash\, \left\{\mathbf{v}^{1}\right\} \right.$ is to the right of~$\mathbf{v}^{1}$. 
        \item The small angle point~$\mathbf{v}^{1}$ is contained in some~$R^{\rho}\left(\mathbf{y}, \mathbf{z}\right) \subset R_{2, 1}^{\rho}$,~$y_{1} < z_{1}$. 
    \end{enumerate}
    If~$X$ is exactly covered by~$\mathscr{A}_{2}^{\rho}$ (or any covering of radius-$\rho$ disks that contains~$\mathscr{A}_{2}^{\rho}$), then~$X$ is \textit{also} exactly covered by~$\left. \mathscr{A}_{2}^{\rho} \,\right\backslash \left\{D_{\mathbf{y}}\right\}$ (Figure~\ref{fig: redundant disk 2}). 
\end{proposition}

\begin{proof}
    See Appendix~\ref{subsec: proofs, redundant disk}. 
\end{proof}

\begin{lemma}\label{lem: f_r}
    Let~$X \in \mathscr{X}'$ with~$\lvert X \rvert \leq 17$. If~$\conv X$ is a triangle or a quadrilateral, then~$X$ can be exactly covered. 
\end{lemma}

\begin{proof}
    Since~$\conv X$ has three or four sides, at least one vertex (after relabeling) has interior angle~$\angle \mathbf{v}^{2} \mathbf{v}^{1} \mathbf{v}^{3} \leq \frac{\pi}{2}$. Without loss of generality, we may flip and rotate~$X$ such that the hypotheses 1--3 of Proposition~\ref{prop: a diagonal lens gives a redundant disk} are satisfied. 

    Suppose that~$\mathbf{v}^{1} \in R^{\rho}\left(\mathbf{y}, \mathbf{z}\right) \subset R_{2, 1}^{\rho}$ (hypothesis 4). This point is covered by two disks~$B_{\mathbf{y}}$ and~$B_{\mathbf{z}}$ of~$\mathscr{A}_{2}^{\rho}$. However, the conclusion of Proposition~\ref{prop: a diagonal lens gives a redundant disk} states that~$X$ is exactly covered by the smaller covering~$\left. \mathscr{A}_{2}^{\rho} \,\right\backslash \left\{D_{\mathbf{y}}\right\}$. 

    Now suppose that~$\mathbf{v}^{1} \notin R_{2, 1}^{\rho}$. This case reduces to the case of a generalized boundary point in the proof of Lemma~\ref{lem: parametric lower bound}. 
    
    Let~$H = \VR_{A_{2}}\left(\mathbf{0}\right)$. The following exclusions apply (cf. the proof of Lemma~\ref{lem: parametric lower bound}): 
    \begin{enumerate}
        \item The small angle point~$\mathbf{v}^{1} \in X$ cannot be in~$R_{2, 2}^{\rho}$, so is in~$R_{0}^{\rho}$,~$R_{1}^{\rho}$, or~$R_{2, 1}^{\rho}$. 
        \item Three (or more) additional generalized boundary points~$\mathbf{v}^{2}, \mathbf{v}^{3}, \mathbf{b}^{4}$ of~$X$ exist by Lemmas~\ref{lem: large triangles have four generalized boundary points}~and~\ref{lem: medium triangles have four generalized boundary points} (including the two other vertices mentioned above). Each of these points cannot be in~$R_{2}^{\rho} = R_{2, 1}^{\rho} \cup R_{2, 2}^{\rho}$, so is in~$R_{0}^{\rho}$ or~$R_{1}^{\rho}$. 
        \item Each of the~$n - 4$ remaining points of~$X$ cannot be in~$R_{0}^{\rho}$ or~$R_{2}^{\rho}$, so is in~$R_{1}^{\rho}$. 
    \end{enumerate}
    The sets of forbidden translation vectors are~$-\mathbf{v}^{1} + R_{2, 2}^{\rho}$;~$-\mathbf{v}^{2} + R_{2}^{\rho}$,~$-\mathbf{v}^{3} + R_{2}^{\rho}$, and~$-\mathbf{b}^{4} + R_{2}^{\rho}$; and~$-\mathbf{x} + \left(R_{0}^{\rho} \cup R_{2}^{\rho}\right)$ for~$\mathbf{x} \in X \left\backslash\, \left\{\mathbf{v}^{1}, \mathbf{v}^{2}, \mathbf{v}^{3}, \mathbf{b}^{4}\right\} \right.$ respectively. 
    
    So if~$n$ satisfies 
    \[
        \area\left(H\right) > \area\left(R_{2, 2}^{\rho} \cap H\right) + 3 \area\left(R_{2}^{\rho} \cap H\right) + \left(n - 4\right) \left(\area\left(R_{0}^{\rho} \cap H\right) + \area\left(R_{2}^{\rho} \cap H\right)\right),
    \]
    then~$\mathscr{A}_{2}^{\rho}$ can be translated to meet conditions 1, 2, and 3, and so there is an exact cover~$\mathscr{D}$ of~$X$. Accordingly, for~$\rho \in \left[1, \sqrt{6} - \sqrt{2}\right]$ define 
    \[
        f_{\text{r}}\left(\rho\right) \coloneqq \frac{\area\left(H\right) - 3 \area\left(R_{2}^{\rho} \cap H\right) - \frac{1}{3} \area\left(R_{2, 2}^{\rho} \cap H\right)}{\area\left(R_{0}^{\rho} \cap H\right) + \area\left(R_{2}^{\rho} \cap H\right)} + 4.
    \]
    If~$\left\vert X \right\vert \leq \left\lceil f_{\text{r}} \left(\rho\right) \right\rceil - 1$ points then~$X$ can be exactly covered by~$\mathscr{D}$. This cover contains all disks of~$\mathscr{A}_{2}^{\rho}$ except for possibly one redundant disk. 
    
    For any~$\mathbf{x} \in A_{2}$, the disk~$D_{\mathbf{x}}$ contains four distinct intersections 
    \[
        R^{\rho}\left(\begin{pmatrix}x_{1} \\ x_{2}\end{pmatrix}, \begin{pmatrix}x_{1} \pm \sqrt{3} \\ x_{2} \pm 1\end{pmatrix}\right) \subset R_{1, 2}
    \]
    and two distinct intersections 
    \[
        R^{\rho}\left(\begin{pmatrix}x_{1} \\ x_{2}\end{pmatrix}, \begin{pmatrix}x_{1} \\ x_{2} \pm 2\end{pmatrix}\right) \subset R_{2, 2}^{\rho},
    \]
    all of identical area. Hence~$\area\left(R_{2, 2}^{\rho} \cap H\right) = \frac{1}{3} \area\left(R_{2}^{\rho} \cap H\right)$. Using the area calculations from the proof of Lemma~\ref{lem: parametric lower bound} and inspection, we obtain 
    \[
        \max_{\rho\in\left[1,\frac{2}{\sqrt{3}}\right]}f_{\text{r}}\left(\rho\right)=17.082... \quad \text{at} \quad \rho=1.027... \, \text{.} \qedhere
    \]
\end{proof}

Finally, we prove our lower bound of~$\widehat{\sigma}_{2} \geq 17$. 

\begin{proof}[Proof of Theorem~\ref{thm: lower bound}]
    Let~$X \in \mathscr{X}'$ with~$\lvert X \rvert \leq 17$. 
    \begin{enumerate}
        \item If~$\conv X$ has three or four sides, then~$X$ can be exactly covered by Lemma~\ref{lem: f_r}. 
        \item If~$\conv X$ has five or more sides, then~$X$ has at least five generalized boundary points, so~$X$ can be exactly covered by the inequality~$\widehat{\sigma}_{2}\left(5\right) \geq 17$ from either of the lower bounds~\ref{eq: lower bound: AHIU + Kozma} or~\ref{eq: lower bound: Inaba + Kozma + parametric}. 
    \end{enumerate}
\end{proof}


\section{An upper bound}\label{sec: upper bound}

\begin{definition}\label{def: distance}
    Let~$X \subset \R^2$ be a non-empty set. The \textit{distance of a point~$\mathbf{y} \in \R^2$ to~$X$} is defined by
    \begin{equation}
        \dist\left(\mathbf{y},X\right) \coloneqq \inf \{\norm{\mathbf{y}-\mathbf{x}}\mid \mathbf{x} \in X\}
    \end{equation}
    and the~\textit{$\varepsilon$-extension} of~$X$ is given by
    \begin{equation}
        X_\varepsilon \coloneqq \left\{\mathbf{y} \in \R^2 \mid \dist\left(\mathbf{y},X\right)\leq\varepsilon\right\}.
    \end{equation}
    We say that~$X$ is an~\textit{$\varepsilon$-net} of~$M\subset \R^2$ if~$M \subset X_\varepsilon$.
\end{definition}
	
\begin{definition}\label{def: epsilon-blocker}
    Let~$M \subset \R^2$ and~$\varepsilon > 0$. We say that~$M$ is an~\textit{$\varepsilon$-blocker} if every~$\varepsilon$-net~$X \in \mathscr{X}$ of~$M$ does not have an exact cover.
\end{definition}
	
We recall that the covering number~$N\left(M,\varepsilon\right)$ of a set~$M\subset \R^2$ is the minimal cardinality of an~$\varepsilon$-net of~$M$. The following statement is a direct consequence of Definition~\ref{def: epsilon-blocker}.
	
\begin{proposition}\label{prop:epsilon_blocker_bound}
    Let~$M \subset \R^2$ be an~$\varepsilon$-blocker. Then~$\widehat{\sigma}_{2} < N\left(M,\varepsilon\right)$. ~$\qed$
\end{proposition}

Our upper bound on~$\widehat{\sigma}_{2}$ follows from the following result, which asserts that every open disk of radius~$R>1$ is an~$\varepsilon$-blocker for a suitably chosen~$\varepsilon>0$. 

\begin{proposition} \label{prop:concrete_no_instances}
    Let~$\varepsilon \in \left(0,7-\sqrt{48} \approx 0.0718\right]$ and
    \begin{equation}
        R \geq \frac{3}{2}\left(1+\varepsilon\right)-\frac{1}{2}\sqrt{1-14\varepsilon+\varepsilon^2}.
    \end{equation} 
    Then~$D_{\mathbf{0},R}$ is an~$\varepsilon$-blocker.
\end{proposition}


The proof of Proposition~\ref{prop:concrete_no_instances} is based on properties of Voronoi diagrams. For the following definition and an overview of Voronoi diagrams, we refer to~\cite[Ch.~6]{JoswigTheobald2013}.
\begin{definition}\label{def: Voronoi region and Voronoi diagram}
    Let~$P \in \mathscr{X}$. The \textit{Voronoi region of~$p \in P$} is given by
  \begin{equation}
      \VR_P(\mathbf{p}) \coloneqq \left\{\mathbf{x} \in \R^2 \,\middle|\, \norm{\mathbf{x}-\mathbf{p}} \leq \norm{\mathbf{x}-\mathbf{q}} \text{ for all } \mathbf{q} \in P\right\}.
  \end{equation}
  It can be shown that every Voronoi region is a polyhedron and that the intersection of two Voronoi regions is a face of both. Therefore, the set of Voronoi regions
  \begin{equation}
      \left\{ \VR_P(\mathbf{p}) \subset \R^2 \,\middle|\, \mathbf{p} \in P\right\}
  \end{equation} 
  forms the set of maximal cells of a polyhedral complex, which is called the \textit{Voronoi diagram of~$P$} and is denoted by~$\VD_P$.
\end{definition}

We say that a set~$P \in \mathscr{X}$ \textit{induces} an exact cover of~$X \in \mathscr{X}$ if~$\mathscr{D}_P$ is an exact cover of~$X$. The following Lemma~asserts that a set~$P \in \mathscr{X}$ that induces an exact cover of~$X \in \mathscr{X}$ has ``almost spherical'' Voronoi regions in the~$\varepsilon$-extension of~$X$.

\begin{lemma} \label{lemma_VR_almost_spherical}
    Let~$\varepsilon \in (0, 1)$. If~$P \in \mathscr{X}$ induces an exact cover of~$X\in\mathscr{X}$, then for every~$\mathbf{p} \in P$ we have
    \begin{equation}
        \left(D_{\mathbf{p}, 1-\varepsilon} \cap X_\varepsilon\right) \subset \left(\VR_P(\mathbf{p}) \cap X_\varepsilon\right) \subset D_{\mathbf{p}, 1+\varepsilon}.
    \end{equation}
\end{lemma}

\begin{proof}
    Let~$\mathbf{y} \in D_{\mathbf{p}, 1-\varepsilon} \cap X_\varepsilon$. By definition of~$X_\varepsilon$, there exists a point~$\mathbf{x} \in X$ with~$\norm{\mathbf{x}-\mathbf{y}} \leq \varepsilon$. Since~$\mathbf{y} \in D_{\mathbf{p}, 1-\varepsilon}$, we have~$\mathbf{x} \in D_{\mathbf{p}}$. Let~$\mathbf{q} \in P$ with~$\mathbf{y} \in \VR_P(\mathbf{q})$. Because~$\norm{\mathbf{y}-\mathbf{q}} \leq \norm{\mathbf{y}-\mathbf{p}}$ and~$\mathbf{y} \in D_{\mathbf{p}, 1-\varepsilon}$, we also have~$\mathbf{x} \in D_{\mathbf{q}}$. Using the fact that~$\mathscr{D}_P$ is an exact cover of~$X$, it follows that~$\mathbf{p}=\mathbf{q}$.
    
    Now let~$\mathbf{y} \in \VR_P(\mathbf{p}) \cap X_\varepsilon$. Again, there is an~$\mathbf{x} \in X$ with~$\norm{\mathbf{x}-\mathbf{y}} \leq \varepsilon$. Because~$\mathscr{D}_P$ is an exact cover of~$X$, there exists a~$\mathbf{q} \in P$ with~$\mathbf{x} \in D_{\mathbf{q}}$ and hence we have
    \begin{equation*}
        1 > \norm{\mathbf{x}-\mathbf{q}} \geq \norm{\mathbf{q}-\mathbf{y}} - \norm{\mathbf{x}-\mathbf{y}} \geq \norm{\mathbf{p}-\mathbf{y}} - \varepsilon. \qedhere
    \end{equation*}
\end{proof}

The basic idea of our proof is that three ``almost spherical'' Voronoi regions cannot meet in a vertex~$\mathbf{v}$ of~$\VD_P$. The following Lemma~establishes the existence of such a vertex.

\begin{lemma} \label{lemma_vertex_in_disk}
    Let~$\varepsilon \in \left(0, \frac{1}{3}\right)$ and~$X \in \mathscr{X}$ be an~$\varepsilon$-net of~$D_{\mathbf{0},1+\varepsilon}$. If~$P \in \mathscr{X}$ induces an exact cover of~$X$, then~$\VD_P$ has a vertex~$\mathbf{v} \in D_{\mathbf{0},1+\varepsilon}$.
\end{lemma}

\begin{proof}
    Because~$X$ is an~$\varepsilon$-net of~$D_{\mathbf{0},1+\varepsilon}$, there exists a point~$\mathbf{x} \in X \cap D_{\mathbf{0},\varepsilon}$. Moreover, since~$\mathscr{D}_P$ is an exact cover of~$X$, there exists a~$\mathbf{p} \in P$ with~$\mathbf{x} \in D_{\mathbf{p}}$ and hence~$\mathbf{p} \in D_{\mathbf{0},1+\varepsilon}$. Let~$L \subset \R^2$ be the perpendicular bisector of~$\mathbf{p}$ and~$\mathbf{0}$. Then the set~$L \cap \bd D_{\mathbf{0},1+\varepsilon}$ contains exactly two points~$\mathbf{y}_1,\mathbf{y}_2$, both of which have distance~$1+\varepsilon$ to~$\mathbf{p}$ (see Figure~\ref{fig: vertex_in_disk}). Let~$A$ be the circular arc of~$\bd D_{\mathbf{0},1+\varepsilon}$ that goes from~$\mathbf{y}_1$ to~$\mathbf{y}_2$ and is on the opposite site of~$\mathbf{p}$. All points in~$A$ have distance at least~$1+\varepsilon$ to~$\mathbf{p}$. By Lemma~\ref{lemma_VR_almost_spherical}, we have~$A \subset \R^2 \setminus \VR_{P}(\mathbf{p})$. If~$\VD_P$ has no vertex in~$D_{\mathbf{0},1+\varepsilon}$, then the line segment~$\conv\{\mathbf{y}_1,\mathbf{y}_2\}$ is contained in~$\R^2 \setminus \VR_{P}(\mathbf{p})$. 
    In particular, the point~$\frac{\mathbf{y}_1+\mathbf{y}_2}{2} \in X_\varepsilon$ is contained in~$\R^2 \setminus \VR_{P}(\mathbf{p})$. Then, by Lemma~\ref{lemma_VR_almost_spherical}, we get~$\frac{\mathbf{y}_1+\mathbf{y}_2}{2} \notin D_{\mathbf{p}, 1 - \varepsilon}$. But this is a contradiction to
    \begin{equation*}
        \left\Vert \mathbf{p}-\frac{\mathbf{y}_1+\mathbf{y}_2}{2} \right\Vert = \frac{1}{2}\norm{\mathbf{p}-\mathbf{0}}<\frac{1+\varepsilon}{2}<\frac{2}{3}<1-\varepsilon. \qedhere
    \end{equation*}

    \begin{figure}[t]
    \centering
    \includegraphics{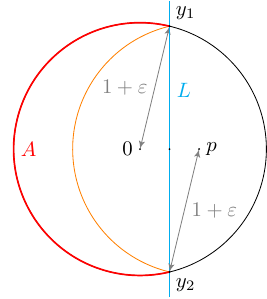}
    \caption{Since~$A$ and~$\VR_P(\mathbf{p})$ are disjoint,~$\VR_P$ either has a vertex in~$D_{\mathbf{0},1+\varepsilon}$ or it does not contain the midpoint of~$\mathbf{0}$ and~$\mathbf{p}$, leading to a contradiction to Lemma~\ref{lemma_VR_almost_spherical}.}
    \label{fig: vertex_in_disk}
\end{figure}
\end{proof}

\begin{proof}[Proof of Proposition~\ref{prop:concrete_no_instances}] 
     Let~$X \in \mathscr{X}$ be an~$\varepsilon$-net of~$D_{\mathbf{0},R}$. We assume towards a contradiction that there exists an exact cover~$\mathscr{D}_P$ of~$X$. Because~$R>1+\varepsilon$,~$X$ is also an~$\varepsilon$-net of~$D_{\mathbf{0},1+\varepsilon}$.
     By Lemma~\ref{lemma_vertex_in_disk},~$\VD_P$ has a vertex~$\mathbf{v} \in D_{\mathbf{0},1+\varepsilon}$, so there are three distinct points~$\mathbf{p}_1,\mathbf{p}_2,\mathbf{p}_3 \in P$ with
     \begin{equation}
        \mathbf{v} \in \textstyle\VR_P(\mathbf{p}_1) \cap  \VR_P(\mathbf{p}_2) \cap \VR_P(\mathbf{p}_3).
     \end{equation}
     By Lemma~\ref{lemma_VR_almost_spherical}, we have
     \begin{equation} \label{eq_prop_concrete_no_instance_1_epsilon_bound}
        1+\varepsilon > \norm{\mathbf{v}-\mathbf{p}_1} = \norm{\mathbf{v}-\mathbf{p}_2} =\norm{\mathbf{v}-\mathbf{p}_3}.
     \end{equation}
     Without loss of generality, we assume that the smaller angle between the vectors~$\mathbf{p}_1-\mathbf{v}$ and~$\mathbf{p}_2-\mathbf{v}$ is at most~$\frac{2\pi}{3}$ (see Figure~\ref{fig: epsilon_blocker_bound}). Because~$\varepsilon\leq 7-\sqrt{48}$, we have~$1-14\varepsilon+\varepsilon^2\geq 0$ and hence we can set
     \begin{equation}
        \mathbf{y} \coloneqq \mathbf{v}+\left(\frac{1+\varepsilon}{2}-\frac{1}{2}\sqrt{1-14\varepsilon+\varepsilon^2}\right) \frac{\frac{\mathbf{p}_1+\mathbf{p}_2}{2}-\mathbf{v}}{\left\Vert \frac{\mathbf{p}_1+\mathbf{p}_2}{2}-\mathbf{v} \right\Vert}.
     \end{equation}

    \begin{figure}[h!]
        \centering
        \includegraphics{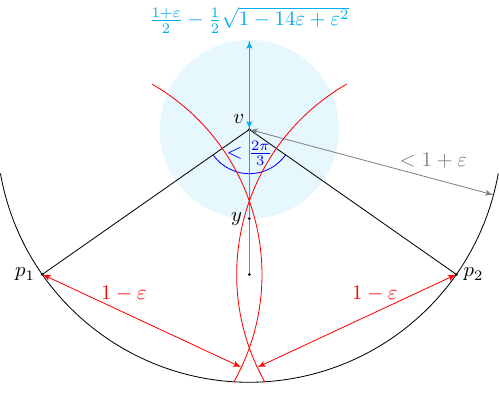}
        \caption{The point~$\mathbf{y}$ is contained in~$D_{\mathbf{p}_1,1-\varepsilon} \cap D_{\mathbf{p}_2,1-\varepsilon} \cap D_{\mathbf{0},R}$, which leads to a contradiction to Lemma~\ref{lemma_VR_almost_spherical}.}
        \label{fig: epsilon_blocker_bound}
    \end{figure}
     
     Let~$i \in \{1,2\}$. We denote the smaller angle between the vectors~$\mathbf{p}_i-\mathbf{v}$ and~$\mathbf{y}-\mathbf{v}$ by~$\theta$. By the law of cosines, we have
     \begin{equation}
        \norm{\mathbf{y}-\mathbf{p}_i}^2 = \norm{\mathbf{v}-\mathbf{p}_i}^2 + \norm{\mathbf{y}-\mathbf{v}}^2 - 2\cdot\norm{\mathbf{v}-\mathbf{p}_i} \cdot \norm{\mathbf{y}-\mathbf{v}} \cdot\cos(\theta).
    \end{equation}
    Since~$\theta \in \left[0,\frac{\pi}{3}\right]$, we have~$\cos(\theta) \geq \frac{1}{2}$. Using~\eqref{eq_prop_concrete_no_instance_1_epsilon_bound} and the fact that~$\norm{\mathbf{v}-\mathbf{p}_i} \geq \norm{\mathbf{y}-\mathbf{v}}$, we obtain
    \begin{align*}
        \norm{\mathbf{y}-\mathbf{p}_i}^2&\leq \norm{\mathbf{v}-\mathbf{p}_i}^2 + \norm{\mathbf{y}-\mathbf{v}}^2 - \norm{\mathbf{v}-\mathbf{p}_i} \cdot \norm{\mathbf{y}-\mathbf{v}}\\
        &= \norm{\mathbf{v}-\mathbf{p}_i}\cdot(\norm{\mathbf{v}-\mathbf{p}_i}- \norm{\mathbf{y}-\mathbf{v}}) + \norm{\mathbf{y}-\mathbf{v}}^2\\
        &<(1+\varepsilon) \left((1+\varepsilon)-\left(\tfrac{1+\varepsilon}{2}-\frac{1}{2}\sqrt{1-14\varepsilon+\varepsilon^2}\right)\right) +\left(\tfrac{1+\varepsilon}{2}-\frac{1}{2}\sqrt{1-14\varepsilon+\varepsilon^2}\right)^2\\
        &=(1-\varepsilon)^2.
    \end{align*}
    
    Because~$\mathbf{y} \in D_{\mathbf{0},R}$, there exists a point~$\mathbf{x} \in X$ with~$\norm{\mathbf{x}-\mathbf{y}}\leq \varepsilon$. By the triangle inequality, we have~$\mathbf{x} \in D_{\mathbf{p}_1} \cap D_{\mathbf{p}_2}$, which contradicts the assumption that~$\mathscr{D}_P$ is an exact cover of~$X$.
\end{proof}

We now obtain Theorem~\ref{thm: upper bound} as a corollary to Proposition~\ref{prop:epsilon_blocker_bound} by setting~$\varepsilon \coloneqq 7-\sqrt{48}$ and~$R\coloneqq \frac{3}{2}\left(1+\varepsilon\right) \approx 1.608$.

\begin{proof}[Proof of Theorem~\ref{thm: upper bound}]
    We set~$\varepsilon \coloneqq 7-\sqrt{48}$ and~$R\coloneqq\frac{3}{2} (1+\varepsilon)$. Using Proposition~\ref{prop:epsilon_blocker_bound}, we obtain that~$D_{\mathbf{0},R}$ is an~$\varepsilon$-blocker. Now Proposition~\ref{prop:epsilon_blocker_bound} asserts that
    \begin{equation}
        \widehat{\sigma}_{2} < N(D_{\mathbf{0},R},\varepsilon).
    \end{equation}
    It is well-known that the hexagonal lattice~$A_2$ has covering radius~$\frac{2}{\sqrt{3}}$. In particular, we have
    \begin{equation}
        \R^2 = \frac{\varepsilon\sqrt{3}}{2}A_2 + D_{\mathbf{0},\varepsilon}.
    \end{equation}
    This implies that for every~$\mathbf{y} \in \R^2$, the set
    \begin{equation} 
        X(\mathbf{y}) \coloneqq \left(\frac{\varepsilon\sqrt{3}}{2}A_2+\mathbf{y}\right) \cap D_{\mathbf{0}, R+\varepsilon}
    \end{equation}
    is an~$\varepsilon$-net of~$D_{\mathbf{0},R}$.

    Using a computer, one can verify that
    \begin{equation*}
        \widehat{\sigma}_{2} < \left|X\left(\begin{pmatrix}\hphantom{+}0.035\\-0.055\end{pmatrix}\right)\right|=657. \qedhere
    \end{equation*}
\end{proof}


\section{Conclusion}

Our main result (Theorems~\ref{thm: lower bound} and~\ref{thm: upper bound}) is~$17 \leq \widehat{\sigma}_{2} \leq 656$. An approach for improving the upper bound could be to search for small~$\varepsilon$-nets of~$\varepsilon$-blockers and to use Proposition~\ref{prop:epsilon_blocker_bound}.

The problem of finding~$\widehat{\sigma}_{2}$ admits generalizations to all dimensions~$d \geq 1$ and convex bodies~$K \subset \mathbb{R}^{d}$. Let~$\sigma\left(K\right)$ and~$\widehat{\sigma}\left(K\right)$ be the largest~$n$ such that any~$n$-point set in~$\mathbb{R}^{d}$ can be covered by disjoint translates of~$K$ or exactly covered by translates of~$K$, respectively (and write~$\sigma_{d}$ and~$\widehat{\sigma}_{d}$ if~$K = B^{d}$). Some of our methods, such as the Extension Argument and the parameterized family, have counterparts for other bodies~$K$, but our other methods do not necessarily generalize. The bulldozer needs to be modified if~$K$ is not strictly convex. The proof of Lemma~\ref{lem: medium triangles have four generalized boundary points} is dependent on the geometry of~$B^{2}$, and the redundant disk method also uses specific properties of~$A_{2}$. 

Sphere packings are mostly empty space in high dimensions. Blichfeldt's upper bound of~$\frac{d + 2}{2} \cdot 2^{-\frac{1}{2} d}$ for the maximum packing density (\cite{Blichfeldt1929}, or see Section 6.1 of~\cite{Zong1999}) drops to less than or equal to~$0.5$ for~$d \geq 6$. The density of the densest known packing in~$d = 5$ is also below~$0.5$ (see Table 1.2 in Chapter 1 of~\cite{ConwaySloane1999}, or~\cite{NebeSloane2012}). Therefore, we cannot hope to cover many points by translating a dense packing of unit balls as in Inaba's proof~\cite{Inaba2008_2} (see Appendix~\ref{sec: Inaba proof}). One possible strategy for ``medium'' dimensions around~$5$--$10$ is to choose one of several packings based on the arrangement of~$X$. 

With regard to lines of further research, we mention the computational complexity of disk covering. Considering the algorithmic issues that were discussed in Subsection~\ref{subsec: relation between exact covering and exact hitting}, it is natural to ask the following question: Given~$X \in \mathscr{X}$, is it NP-hard to decide whether~$X$ has an exact cover? Ashok, Basu Roy, and Govindarajan (2020)~\cite{ashokEtAl20} showed that it is NP-hard to decide the following problem: Given a finite set~$\mathscr{R}$ of unit squares and given an~$X \in \mathscr{X}$, is there a subset~$\mathscr{R}'\subset\mathscr{R}$ that exactly covers~$X$? Their proof can be easily adapted for a given family~$\mathscr{D}$ of unit disks. It might also be interesting to study the computational complexity if the number of disks in the exact cover is specified. 

\subsection*{Acknowledgments}

Thanks to Michaela Borzechowski, L\'aszl\'o Kozma, and G\"unter Rote for their useful feedback and helpful discussions. 

\bibliography{bibliography}

\begin{thebibliography}{10}

\bibitem{AloupisHearnIwasawaUehara2012}
Greg Aloupis, Robert~A. Hearn, Hirokazu Iwasawa, and Ryuhei Uehara.
\newblock Covering points with disjoint unit disks.
\newblock In {\em Canadian Conference on Computational Geometry}, 2012.
\newblock URL: \url{https://api.semanticscholar.org/CorpusID:16280099}.

\bibitem{ashokEtAl20}
Pradeesha Ashok, Aniket Basu~Roy, and Sathish Govindarajan.
\newblock Local search strikes again: {PTAS} for variants of geometric covering
  and packing.
\newblock {\em J. Comb. Optim.}, 39(2):618--635, 2020.
\newblock \href {https://doi.org/10.1007/s10878-019-00432-y}
  {\path{doi:10.1007/s10878-019-00432-y}}.

\bibitem{BetkeHenkWills1994}
Ulrich Betke, Martin Henk, and J\"org Wills.
\newblock Finite and infinite packings.
\newblock {\em {Journal f\"ur die reine und angewandte Mathematik}},
  1994:165--192, 01 1994.
\newblock \href {https://doi.org/10.1515/crll.1994.453.165}
  {\path{doi:10.1515/crll.1994.453.165}}.

\bibitem{Blichfeldt1929}
Hans~Frederick Blichfeldt.
\newblock The minimum value of quadratic forms, and the closest packing of
  spheres.
\newblock {\em Mathematische Annalen}, 101(1):605--608, 1929.
\newblock \href {https://doi.org/10.1007/BF01454863}
  {\path{doi:10.1007/BF01454863}}.

\bibitem{Boeroeczky2004}
K\'aroly B\"or\"oczky, Jr.
\newblock {\em Finite Packing and Covering}.
\newblock Cambridge Tracts in Mathematics. Cambridge University Press, 2004.

\bibitem{ConwaySloane1999}
John~H. Conway and Neil J.~A. Sloane.
\newblock {\em Sphere Packings, Lattices and Groups}, volume 290 of {\em
  Grundlehren der mathematischen Wissenschaften}.
\newblock Springer-Verlag New York, Inc., 3rd edition, 1999.

\bibitem{Elser2011}
Veit Elser.
\newblock Packing-constrained point coverings, 2011.
\newblock \href {http://arxiv.org/abs/1101.3468} {\path{arXiv:1101.3468}}.

\bibitem{FejesToth1975}
L\'aszl\'o Fejes~T\'oth.
\newblock Research problem 13.
\newblock {\em Periodica Mathematica Hungarica}, 6(2):197--199, 1975.

\bibitem{HenkWills2020}
Martin Henk and J{\"o}rg~M. Wills.
\newblock Packings, sausages and catastrophes.
\newblock {\em Beitr{\"a}ge zur Algebra und Geometrie / Contributions to
  Algebra and Geometry}, 2020.
\newblock \href {https://doi.org/10.1007/s13366-020-00502-x}
  {\path{doi:10.1007/s13366-020-00502-x}}.

\bibitem{AMDBulldozer2007}
Phil Hester and Bob Drebin.
\newblock 2007 {Technology Analyst Day}, Jul 2007.
\newblock URL:
  \url{http://web.archive.org/web/20081209041620/http://www.amd.com/us-en/assets/content_type/DownloadableAssets/July_2007_AMD_Analyst_Day_Phil_Hester-Bob_Drebin.pdf}
  [cited 2023-10-24].

\bibitem{Inaba2008_1}
Naoki Inaba, 2008.
\newblock URL: \url{http://inabapuzzle.com/hirameki/suuri_4.html}.

\bibitem{Inaba2008_2}
Naoki Inaba, 2008.
\newblock URL: \url{http://inabapuzzle.com/hirameki/suuri_ans4.html}.

\bibitem{JoswigTheobald2013}
Michael Joswig and Thorsten Theobald.
\newblock {\em Polyhedral and Algebraic Methods in Computational Geometry}.
\newblock Universitext. Springer, 2013.

\bibitem{junttilaKaski10}
Tommi Junttila and Petteri Kaski.
\newblock Exact cover via satisfiability: An empirical study.
\newblock In David Cohen, editor, {\em Principles and Practice of Constraint
  Programming -- CP 2010}, pages 297--304. Springer Berlin Heidelberg, 2010.

\bibitem{Knuth2000}
Donald~E. Knuth.
\newblock Dancing links, 2000.
\newblock \href {http://arxiv.org/abs/cs/0011047} {\path{arXiv:cs/0011047}}.

\bibitem{KnuthTAOCP4B2023}
Donald~E. Knuth.
\newblock {\em The Art of Computer Programming}, volume~4B.
\newblock Pearson Education, Inc., 2023.

\bibitem{NebeSloane2012}
Gabriele Nebe and Neil J.~A. Sloane.
\newblock Table of densest packings presently known, Feb 2012.
\newblock URL: \url{math.rwth-aachen.de/~Gabriele.Nebe/LATTICES/density.html}
  [cited 2024-01-28].

\bibitem{OkayamaKiyomiUehara2012}
Yosuke Okayama, Masashi Kiyomi, and Ryuhei Uehara.
\newblock On covering of any point configuration by disjoint unit disks.
\newblock {\em Geombinatorics}, 22(1):14--23, 2012.

\bibitem{Wells1991}
David Wells.
\newblock {\em The Penguin Dictionary of Curious and Interesting Geometry}.
\newblock Penguin Books, 1991.

\bibitem{Winkler2010}
Peter Winkler.
\newblock Puzzled: Solutions and sources.
\newblock {\em Commun. ACM}, 53(9):110, Sep 2010.
\newblock \href {https://doi.org/10.1145/1810891.1810917}
  {\path{doi:10.1145/1810891.1810917}}.

\bibitem{Zong1999}
Chuanming Zong.
\newblock {\em Sphere Packings}.
\newblock Universitext. Springer-Verlag New York, Inc., 1999.

\end{thebibliography}

\newpage
\appendix
\noindent\textbf{\LARGE Appendix}


\section{Proof of Inaba's result}\label{sec: Inaba proof}

Let~$K$ be a convex body in~$\mathbb{R}^{2}$ and~$C \subset \mathbb{R}^{2}$. The set~$C + K \subseteq \mathbb{R}^{2}$ is called a \textit{packing (set)} of~$K$ if \[\left(\mathbf{y} + K\right) \cap \left(\mathbf{z} + K\right) = \emptyset\] for all~$\mathbf{y}, \mathbf{z} \in C$,~$\mathbf{y} \neq \mathbf{z}$. The proportion of~$\mathbb{R}^{2}$ covered by~$C + K$ is called the \textit{(infinite) packing density} of~$C + K$, denoted by~$\delta\left(K, C\right)$. 

\begin{theorem}[Inaba, 2008~\cite{Inaba2008_2}]\label{thm:inaba}
We have~$\sigma_{2} \geq 10$. 
\end{theorem}

\begin{proof}
    Let~$X \in \mathscr{X}$. Place the disk centers at the points of the hexagonal lattice~$A_{2}$, where the minimum distance between distinct points is~$2$. This packing~$A_{2} + B^{2}$ has a density of 
    \[
        \delta\left(B^{2}, A_{2}\right) = \frac{\pi}{2 \sqrt{3}} \approx 0.9069.
    \]
    Let~$\mathscr{A}_{2} \coloneqq \left\{ D_{\mathbf{c}}\subset\mathbb{R}^{2} \,\middle|\, \mathbf{c} \in A_{2}\right\}$ be the collection of disks with centers at the points of the packing set~$A_{2}$. We want to find a translation vector~$\mathbf{t} \in \mathbb{R}^{2}$ so that the translated collection~$\mathbf{t} + \mathscr{A}_{2} = \left\{\mathbf{t} + D \,\middle|\, D \in \mathscr{A}_{2}\right\}$ of disks covers all points in~$X$. Each point~$\mathbf{x} \in X$ defines a set
    \begin{equation}
        F_{\mathbf{x}} \coloneqq \left\{ \mathbf{t} \in \mathbb{R}^{2} \,\middle|\, \mathbf{x} \notin D \text{ for all } D \in \mathbf{t} + \mathscr{A}_{2}\right\} \label{eq: forbidden set}
    \end{equation}
    of ``forbidden'' translation vectors~$\mathbf{t}$ that result in a translation not covering~$\mathbf{x}$. The proportion of the area of each~$F_{\mathbf{x}}$ in relation to the entire plane is 
    \[
        1 - \delta\left(B^{2}, A_{2}\right) \approx 0.0931 < \frac{1}{10}.
    \] 
    Hence, if~$\lvert X \rvert \leq 10$, the union of all~$F_{\mathbf{x}}$ does not exhaust the entire plane. Now, some translation vector~$\mathbf{t}' \in \mathbb{R}^{2}$ remains that is not part of any~$F_{\mathbf{x}}$,~$\mathbf{x} \in X$, and the corresponding covering~$\mathbf{t}' + \mathscr{A}_{2}$ contains~$X$. Hence~$\sigma_{2} \geq 10$. 
\end{proof}

Inaba's method applies to any lattice~$\Lambda \subset \mathbb{R}^{2}$. The lower bound for~$\sigma_{2}$ obtained by this method is straightforward except for one wrinkle: if~$X \in \mathscr{X}$ with~$\left\vert X \right\vert = n = \frac{1}{1 - \delta\left(B^{2}, \Lambda\right)}$. In this case, we have~$\bigcup_{\mathbf{x} \in X} F_{\mathbf{x}} = \mathbb{R}^{2}$ since each set~$F_\mathbf{x}$ is closed, so no~$\mathbf{t}'$ exists as described above. Hence we are only guaranteed~$\widehat{\sigma}_{2} \geq n - 1$ instead of~$\widehat{\sigma}_{2} \geq n$, for 
\begin{equation}
    \sigma_{2} \geq \left\lceil \frac{1}{1 - \delta\left(B^{2}, \Lambda\right)} \right\rceil - 1. 
    \label{eq: general lower bound}
\end{equation}


\section{Technical proofs}\label{sec: lower bound proofs}

In this section we provide several lengthy or technical proofs that would otherwise cover up the main ideas of Section~\ref{sec: lower bound}. 

\subsection{Technical proofs for the generalized boundary points}\label{subsec: proofs, boundary points}

\begin{proof}[Proof of Proposition~\ref{prop: narrow triangles are small}]
    Write~$T = \conv\left\{\mathbf{v}^{1}, \mathbf{v}^{2}, \mathbf{v}^{3}\right\}$ with the longest side~$\overline{\mathbf{v}^{1} \mathbf{v}^{2}}$. 
    \begin{enumerate}
        \item If~$T$ is a right triangle, then the circumcenter of~$T$ is at the midpoint~$\mathbf{m} \coloneqq \frac{\mathbf{v}^{1} + \mathbf{v}^{2}}{2}$ of the hypotenuse~$\overline{\mathbf{v}^{1} \mathbf{v}^{2}}$. This side has length less than~$2$, so~$R_{T} < 1$ and~$T \subset D_{\mathbf{m}}$. 
        \item If~$T$ is an obtuse triangle, then~$T \subset T_{R}$ for any right triangle~$T_{R}$ with hypotenuse~$\conv\left\{\mathbf{v}^{1}, \mathbf{v}^{2}\right\}$ and which contains~$\mathbf{v}^{3}$. Then by part 1 applied to~$T_{R}$, it follows that~$T \subset T_{R} \subset D_{\mathbf{m}}$. \qedhere
    \end{enumerate}
\end{proof}

\begin{proof}[Proof of Lemma~\ref{lem: three bulldozers cover T}]
    It suffices to show that~$D_{i, j} \cap T \subseteq \BD_{i, j} \cap\, T$. Without loss of generality, we only need to prove the Lemma~for~$i = 1$ and~$j = 2$, and we may assume that~$\mathbf{v}^{1}=\begin{pmatrix}-v \\ \hphantom{+}0\end{pmatrix}$ and~$\mathbf{v}^{2}=\begin{pmatrix}v \\ 0\end{pmatrix}$ as in Definition~\ref{def: bulldozer}. Write 
    \begin{align}
        \BD_{1, 2} & = \left\{\begin{pmatrix}z_{1} \\ z_{2}\end{pmatrix} \in \mathbb{R}^{2} \,\middle|\, \left(z_{1}\right)^{2} + \left(z_{2} + \sqrt{1 - v^{2}}\right)^{2} < 1\right\}, \label{eq: BD formula} \\
        D_{1, 2} & = \left\{\begin{pmatrix}z_{1} \\ z_{2}\end{pmatrix} \in \mathbb{R}^{2} \,\middle|\, \left(z_{1}\right)^{2} + \left(z_{2} + \sqrt{R_{T} - v^{2}}\right)^{2} < R_{T}^{2}\right\}. \label{eq: D formula}
    \end{align}
    Let~$\mathbf{z} \in D_{1, 2} \cap T$. We substitute~\eqref{eq: D formula} into~\eqref{eq: BD formula} for 
    \begin{eqnarray*}
    \left(z_{1}\right)^{2}+\left(z_{2}+\sqrt{1-v^{2}}\right)^{2} & < & R_{T}^{2}-\left(z_{2}+\sqrt{R_{T}^{2}-v^{2}}\right)^{2}+\left(z_{2}+\sqrt{1-v^{2}}\right)^{2}\\
     & = & 1-2z_{2}\left(\sqrt{R_{T}^{2}-v^{2}}-\sqrt{1-v^{2}}\right).
    \end{eqnarray*}
    The rightmost quantity is less than~$1$ since~$z_{2} \geq 0$ and~$R_{T}^{2} > 1$, so~$\mathbf{z} \in \BD_{1, 2} \cap \, T$. 
\end{proof}




\subsection{Technical proofs for the redundant disk method}\label{subsec: proofs, redundant disk}

\begin{proof}[Proof of Proposition~\ref{prop: diagonal lenses are steep}]
    Assume that~$\mathbf{y} = \begin{pmatrix} -\sqrt{3} \\ 1 \end{pmatrix}$ and~$\mathbf{z} = \begin{pmatrix} 0 \\ 0 \end{pmatrix}$. Let~$\mathbf{w}^{\rho}$ be the point of intersection between~$\bd D_{(-\sqrt{3}, 1)^{\mathsf{T}}, \rho}$ and~$\bd D_{\mathbf{0}, \rho}$ with higher~$y$-coordinate (see Figure~\ref{fig: redundant disk slope}); this point exists because~$\rho \geq 1$. We will show that the tangent line~$L^{\rho}\left(\mathbf{w^{\rho}}\right)$ has slope~$\geq 1$. 
    
    \begin{figure}
        \centering
        \includegraphics[width=4.75in]{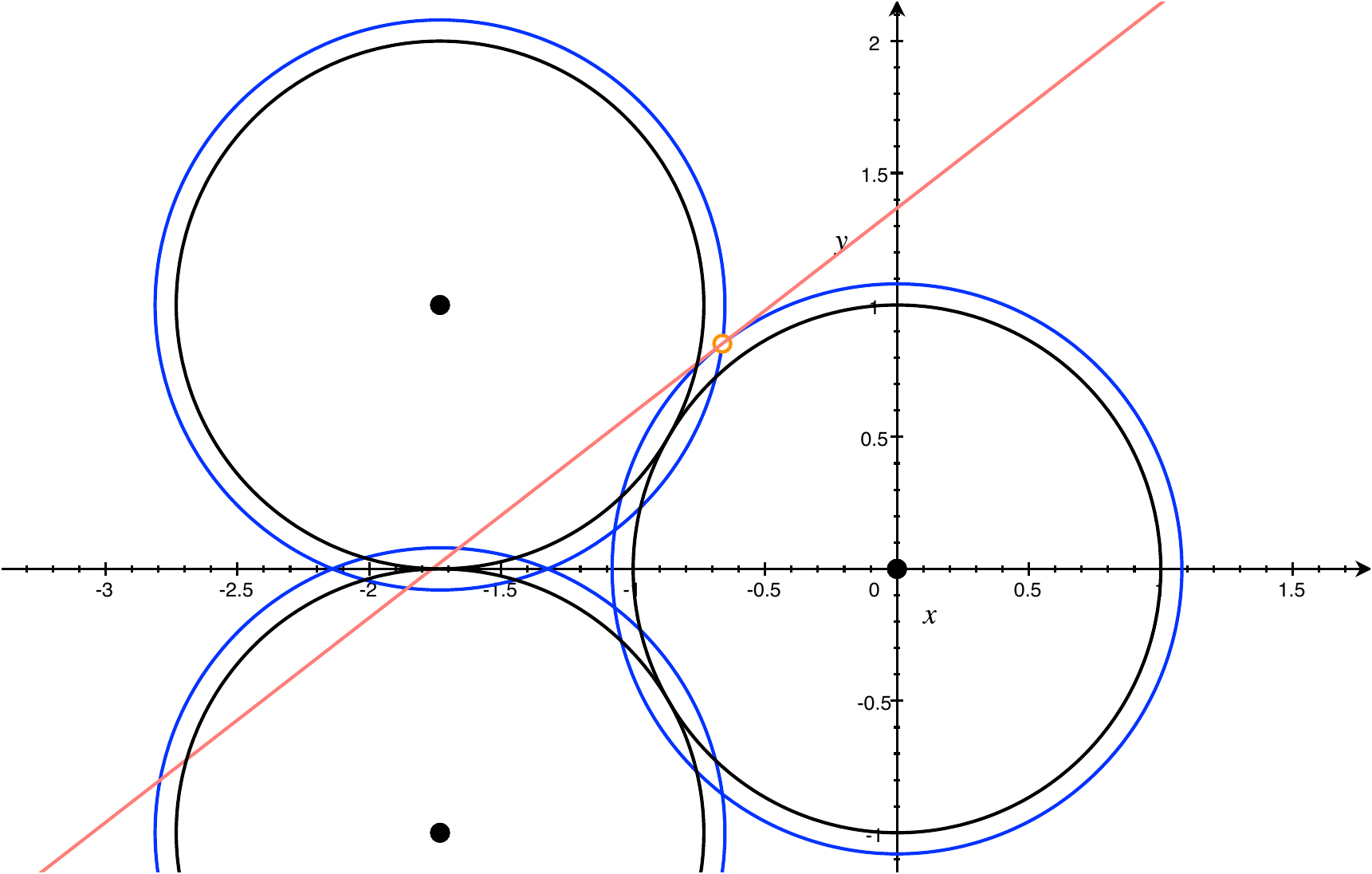}
        \caption{The disks~$D_{\mathbf{w}}$ (black) and~$D_{\mathbf{w}, 1.08}$ (blue) for~$\mathbf{w} = \mathbf{y} = \begin{pmatrix}-\sqrt{3}\\1\end{pmatrix}$ (top left),~$\mathbf{w} = \mathbf{z} = \begin{pmatrix}0\\0\end{pmatrix}$ (lower right), and~$\mathbf{x} = \begin{pmatrix}-\sqrt{3}\\-1\end{pmatrix}$ (bottom left), the point~$\mathbf{w}^{1.08}$ (red), and the line~$L^{1.08}\left(\mathbf{w^{1.08}}\right)$ (red).}
        \label{fig: redundant disk slope}
    \end{figure}

    The general case reduces to this special case by the symmetry of~$A_{2}$ and the circle. 
    
    To calculate the coordinates of the intersection point~$\mathbf{w}^{\rho}$, we parameterize the boundary circles of~$D_{\mathbf{y}}$ and~$D_{\mathbf{z}}$ by 
    \begin{align*}
        \mathbf{p}_{\mathbf{y}}\left(\phi\right) & =
        \begin{pmatrix}
            -\sqrt{3} \\
            1
        \end{pmatrix}
        +\rho
        \begin{pmatrix}
            \hphantom{+}\cos\left(\phi-\frac{2\pi}{3}\right) \\
            -\sin\left(\phi-\frac{2\pi}{3}\right)
        \end{pmatrix}, \\
        \mathbf{p}_{\mathbf{z}}\left(\phi\right) & =\rho
        \begin{pmatrix}
            \cos\phi \\
            \sin\phi
        \end{pmatrix},
    \end{align*}
    for~$\phi \in \left[0, 2 \pi\right]$, so that
    \begin{align*}
        D_{\mathbf{y}} & =\left\{ \mathbf{p}_{\mathbf{y}}\left(\phi\right)\,\middle|\,0\leq\phi\leq2\pi\right\}, \\
        D_{\mathbf{z}} & =\left\{ \mathbf{p}_{\mathbf{z}}\left(\phi\right)\,\middle|\,0\leq\phi\leq2\pi\right\}.
    \end{align*}
    The parameterization of~$D_{\mathbf{z}}$ is as usual, but the parameterization of~$D_{\mathbf{y}}$ is offset and moves clockwise so that~$\mathbf{p}_{\mathbf{y}}\left(\phi\right)$ and~$\mathbf{p}_{\mathbf{z}}\left(\phi\right)$ are always equidistant from the tangent line common to~$D_{\mathbf{y}}$ and~$D_{\mathbf{z}}$. In particular,~$\mathbf{w}^{\rho} = \mathbf{p}_{\mathbf{y}}\left(\phi\right) = \mathbf{p}_{\mathbf{z}}\left(\phi\right)$ for the smallest positive~$\phi$ such that~$\mathbf{p}_{\mathbf{y}}\left(\phi\right) = \mathbf{p}_{\mathbf{z}}\left(\phi\right)$. This equality is equivalent to the system 
    \begin{align*}
        \rho\cos\phi & =-\sqrt{3}+\rho\cos\left(\phi-\frac{2\pi}{3}\right),\\
        \rho\sin\phi & =1-\rho\sin\left(\phi-\frac{2\pi}{3}\right).
    \end{align*}
    The smallest positive solution of either equation is given by
    \[
        \phi = \phi_{\rho} \coloneqq \frac{5 \pi}{6} - \arcsec\left(\rho\right),
    \]
    which yields
    \[
        \mathbf{w}^{\rho} = \mathbf{p}_{\mathbf{z}}\left(\phi_{\rho}\right) = \frac{1}{2}\begin{pmatrix}
        -\sqrt{3} + \sqrt{\rho^{2} - 1} \\ 
        \sqrt{3} \sqrt{\rho^{2} - 1} + 1
        \end{pmatrix}.
    \]
    Hence~$L^{\rho}\left(\mathbf{w^{\rho}}\right)$ is given by the equation 
    \[
        y = \frac{\sqrt{3} - \sqrt{\rho^{2} - 1}}{\sqrt{3} \sqrt{\rho^{2} - 1} + 1} \left(x - \frac{1}{2} \left(-\sqrt{3} + \sqrt{\rho^{2} - 1}\right)\right) + \frac{1}{2} \left(\sqrt{3} \sqrt{\rho^{2} - 1} + 1\right),
    \]
    whose slope is greater than or equal to~$1$ if~$\rho \leq \sqrt{6} - \sqrt{2}$. 
\end{proof}

\begin{proof}[Proof of Proposition~\ref{prop: a diagonal lens gives a redundant disk}]
    Note that~$R^{\rho}\left(\mathbf{y}, \mathbf{z}\right)$ is a subset of the half-space~$S \coloneqq \left\{\mathbf{w} \in \mathbb{R}^{2} \,\middle|\, w_{1} \leq w_{1}^{\rho}\right\}$, where~$\mathbf{w}^{\rho}$ is the rightmost ``vertex'' of~$R^{\rho}\left(\mathbf{y}, \mathbf{z}\right)$ as in the proof of Proposition~\ref{prop: diagonal lenses are steep}. So we may restrict our attention to the subset~$\left(\conv X\right) \cap S$ of~$\conv X$. The interior angle at~$\mathbf{v}^{1}$ is at most~$\frac{\pi}{2}$ and~$\overline{\mathbf{v}^{1} \mathbf{v}^{2}}$ and~$\overline{\mathbf{v}^{1} \mathbf{v}^{3}}$ have equal and opposite slopes, so these slopes have absolute values of at most~$1$. However, by Proposition~\ref{prop: diagonal lenses are steep}, the slope of the upper arc of~$\bd\left(R^{\rho}\left(\mathbf{y}, \mathbf{z}\right)\right)$ is always greater than or equal to~$1$. So, by the Mean Value Inequality and the symmetry of the circle, we have~$\left(\conv X\right) \cap S \subseteq D_{\mathbf{z}}$. 
\end{proof}


\end{document}